\documentclass[a4paper]{article}

\usepackage{packs/standard}
\usepackage{packs/thm}
\usepackage{packs/macros}

\numberwithin{equation}{section}

\begin{document}

\title{%
  The Motivic Picard--Lefschetz Formula
	\\ 
	{\large\itshape Preliminary version}\footnote{%
		Statements whose proofs
		will be expanded are marked with asterisks.%
	}%
}
\author{Ran Azouri, Emil Jacobsen}
\date{\today}
\maketitle

\begin{abstract}%
We prove a motivic enhancement of the classical Picard--Lefschetz formula.
Our proof is completely motivic,
and yields a description of
the motivic nearby cycles at
a quasi-homogeneous singularity,
as well as its monodromy,
in terms of an embedding of projective hypersurfaces.
\end{abstract}

\setcounter{tocdepth}{1}
\tableofcontents

\section{Introduction}
Our main theorem is a motivic enhancement of the Picard--Lefschetz formula.
That is,
in the presence of an isolated singularity
in the special fibre,
we compute the nearby cycles motive
and the associated monodromy.
The nearby cycles functor
is an important tool
in complex geometry,
$\ell$-adic theory,
and motives alike.
It plays a key role in the theory of perverse sheaves,
Hodge modules,
singularities,
etc.
The $\ell$-adic Picard--Lefschetz formula
is used in Deligne's proof of the Weil conjectures.

In order to introduce the setting,
fix a base field $k$,
write $\Affl := \Affl_k$,
and
let
\begin{equation*}
f \colon X \to \Affl
\end{equation*}
be a family of schemes.
Denote the fibre over $0\in\Affl$ by $X_0$.
Ayoub constructs~\cite{ayoub_thesis_2,ayoub_etale}
the \emph{(unipotent) motivic nearby cycles complex}
$\Psi_f\bm1 \in \DA(X_0)$,
in the infinity category of étale motives
with rational coefficients
over $X_0$.\footnote{%
	See Rmk.~\ref{rmk:upsilon}%
}
It agrees with its classical analogues
under the Betti and $\ell$-adic realizations.
Ayoub also constructs~[ibid.] a monodromy operator
\begin{equation*}
N \colon \Psi_f\bm1 \to \Psi_f\bm1(-1),
\end{equation*}
and this too is compatible with its
counterparts after realization.

The Picard--Lefschetz problem concerns the following situation:
assume that the generic fibre of $f \colon X \to \Affl$ is smooth,
and that there is a single, isolated singularity $o$
in the special fibre $X_0$.
The problem of computing
$\Psi_f\bm1$ and its monodromy
then reduces to 
computing the \emph{variation}, which describes the monodromy at the point.
\begin{equation*}
\var \colon o^*\Psi_f\bm1 \to o^!\Psi_f(-1).
\end{equation*}
In the context of singular and $\ell$-adic cohomology,
the classical Picard--Lefschetz formula provides a full answer
when $o$ is assumed to be a quadratic homogeneous singularity.
The $\ell$-adic formula is due to Deligne~\cite[Exp.~XV]{SGA7}.
Illusie~\cite{illusiePL} gives a purely algebraic proof,
without relying on transcendental methods.
Our proof is motivic,
and does not use any realization functors.
We also work with a general, (quasi-)homogeneous singularity.

Let us introduce some more notation.
Assume that $o$ is a homogeneous singularity defined by
a homogeneous polynomial $F$ of degree $r$.
Consider the complementary closed and open immersions
\begin{equation*}
C:= \{T_{n+1} = 0\}
\xhookrightarrow{\; \; i \; \;} D := V_{\P^{n+1}}(F-T_{n+1}^r)
\xhookleftarrow{ \; \; j \; \;} A := \{ T_{n+1} \neq 0 \}
.
\end{equation*}
We use $\alpha$ and $\beta$ to denote
the following natural maps
coming from localization sequences in $\DA(k)$:
\begin{equation*}
h(A) \xlongrightarrow{\alpha} h(C)(-1)[-1]
,
\quad
h(C)[-1] \xlongrightarrow{\beta} h_{\rom{c}}(A)
,
\end{equation*}
where $h(\blank)$ and $h_\rom{c}(\blank)$
denote the cohomological motive
and the cohomological motive with compact support,
respectively.
Our first main theorem is:

\begin{thm}[%
	The general Picard--Lefschetz formula,
	Thm.~\ref{thm:abstractPL}%
]\label{thm:general_intro}
Assume the characteristic of $k$ is different from $2$.
Let $X$ be regular and let
$f \colon X \to \Affl$
be a flat,
quasi-projective morphism.
Assume $f$ is smooth
except for an isolated homogeneous singularity
$o \in X_0$,
defined by a homogeneous polynomial $F$ of degree $r$.
With notation as above,
there are natural equivalences
\begin{equation*}
h(A) \simeq o^*\Psi_f\one, 
\quad
h_{\rom{c}}(A) \simeq o^! \Psi_f \one
,
\end{equation*}
such that the following diagram commutes:
\begin{equation*}\begin{tikzcd}
o^* \Psi_f \one
	\arrow[d, "\sim" {rotate=-90, anchor=north}]
	\arrow[rr, "-r \cdot \var"]
&& o^! \Psi_f \one(-1)
\\
h( {A} )	\arrow[r, "\alpha"]  	&	h(C) (-1)[-1]  \arrow[r, "\beta(-1)"]
&	h_{\rom{c}}(A)(-1)  \arrow[u, "\sim" {rotate=90, anchor=north}]
.
\end{tikzcd}\end{equation*}
\end{thm}

\begin{rmk}
For simplicity,
we only give the homogeneous case here,
but a similar statement
(Thm.~\ref{thm:abstractPL})
is valid in the generality
of quasi-homogeneous singularities
(see Def.~\ref{def:qhomogeneoussingularity}).
\end{rmk}
 
\paragraph{Proof strategy.}
Our proof of the theorem goes through a series of reductions.
For the first step, we show how to
replace $f$ by a semistable family
$g \colon Y \to \Affl$
with two branches,
similar to the approach of Illusie,
obtaining the map $g$ from $f$ by
base-change, blowup, and normalization.
We then have to keep track of how
nearby cycles motive
and the monodromy
changes as we replace $f$ in this way.
The outcome is
a computation of the motives
$o^* \Psi_f \one$
and $o^! \Psi_f \one$
as in the statement of Thm.~\ref{thm:general_intro}.
We also compute the variation,
$\var$,
but only up to
an \emph{a priori} unknown rational number $\lambda$.
This is the content of
\S~\ref{sec:ss_reduction}.

In order to determine $\lambda$,
we first replace the general semistable family $g$
by the one-dimensional family
\begin{equation*}
g \colon \Spec k[t,x,y]/(xy-t) \to \Affl,
\end{equation*}
(Prop.~\ref{prop:reduction_to_1dim}).
We then reduce computing the variation of this family
to computing the monodromy of the Kummer motive
$\mathcal{K} \in \DA(\mathbb{G}_m)$,
as it turns out that $\Psi_g \one$
is closely related to $\Psi_{\id}\mathcal{K}$.
This is the content of \S~\ref{sec:reduction_to_kummer}.

In the last step, we compute this monodromy, i.e.,
\begin{equation*}
N \colon \Psi_{\id}\mathcal{K} \to \Psi_{\id}\mathcal{K}(-1)
.
\end{equation*}
By opening up Ayoub's construction
of the motivic monodromy operator,
we reduce this to a concrete question about the Kummer and logarithm motives.
We then ascertain the value of $\lambda$ to be $-1$,
finishing the proof of
Thm.~\ref{thm:general_intro}.
This is the content of \S~\ref{sec:kummer}.

\paragraph{Quadratic singularities.}
Recall that the Picard--Lefschetz
formula classically
concerns the case where $o$
is a quadratic homogeneous singularity.
Under this assumption,
we can be more precise
by actually computing
the motives of $A$ and $C$
in Thm.~\ref{thm:general_intro}:

\begin{thm}[%
	The quadratic Picard--Lefschetz formula,
	Thm.~\ref{thm:quadratic_PL}%
]\label{thm:quadratic_intro}
Assume $k$ is algebraically closed
of characteristic different from $2$.\footnote{%
	It is probably true for any $k$ of characteristic different from $2$.
  This will be resolved in a later version of this text.%
}
Let $X$ be regular
and let
$f \colon X \to \Affl$
be a flat, quasi-projective morphism of relative dimension $n$,
smooth except for
an isolated non-degenerate quadratic singularity
$o \in X_0$.
There are canonical fibre sequences
\begin{equation*}
\bm1
\to o^*\Psi_f\bm1
\xrightarrow{m_1} \bm1(-\lceil n/2 \rceil)[-n],
\end{equation*}
and
\begin{equation*}
\bm1(-\lfloor n/2 \rfloor)[-n]
\xrightarrow{m_2} o^!\Psi_f\bm1
\to \bm1(-n)[-2n]
.
\end{equation*}
If $n$ is even, $\var$ is the zero map.
If $n=2m+1$ is odd,
then
$\var$ factors as
\begin{equation*}
o^* \Psi_f \one
\xrightarrow{m_1} \bm1(-m-1)[-n]
\xrightarrow{-1} \bm1(-m-1)[-n]
\xrightarrow{m_2} o^! \Psi_f \one(-1)
.
\end{equation*}
\end{thm}

Let the singularity at $o$
be defined by a quadratic homogeneous polynomial $F$
defining a non-degenerate quad\-rat\-ic form.
Recall the varieties $D$, $C$, and $A$ from above.
Then $D$ and $C$ are smooth projective quadrics,
and $A$ is an affine quadric.
We compute their motives based on
Rost's work on the Chow motives of quadrics~\cite{Rost}.
Combined with Thm.~\ref{thm:general_intro},
this yields Thm.~\ref{thm:quadratic_intro}.
This is the content of \S~\ref{sec:quadratic},
the final section.

\paragraph{Historical context.}
The formula originates in complex geometry,
in the work of Picard~\cite{Picard}
for a holomorphic function on a surface,
and Lefschetz~\cite{Lefschetz}
for higher dimensional manifolds.
In~\cite[\'expos\'e XV]{SGA7},
Deligne proves an algebro-geometric formula
in the setting of étale cohomology.
This result plays a crucial role in
his proof of the Weil conjectures.
While Deligne's proof uses transcendental methods,
a purely algebraic proof for the formula
is given by Illusie~\cite{illusiePL}.

\begin{rmk}
It has been brought to our attention that the problem
of motivically enhancing the Picard--Lefschetz formula
was considered independently by Roland Casalis,
a former graduate student of Fr\'ed\'eric D\'eglise,
with no results published.
\end{rmk}

\paragraph{Acknowledgements.}
We would like to thank Joseph Ayoub,
whose ideas, advice and support have been invaluable.
At the start of this project,
both authors were supported
by the SNF project Motives and Algebraic Cycles
(grant no.\ 178729).
The first author also acknowledges support by
the project Foundations of Motivic Real K-Theory
(ERC grant no.\ 949583).
The second author was also supported by
Dan Petersen's Wallenberg Scholar fellowship.


\section{Preliminaries}

\subsection{Notations and conventions}

\begin{notation}\label{notation:fam_over_A1}
Let $\pt$ denote $\Spec(k)$.
We write $\Affl$ for the affine line
$\mathbb{A}^1_k = \Spec k[t]$
over $k$,
and similarly $\Gm \subseteq \Affl$
for the multiplicative group scheme
$\mathbb{G}_{\rom{m},k}$.
Their structure morphisms are denoted by
$p \colon \Affl \to \pt$ and
$q \colon \Gm \to \pt$, respectively.
For a scheme $f \colon X \to \Affl$ over the affine line,
we write $X_\eta := X \times_{\Affl} \Gm$ and $X_\sigma := X \times_{\Affl} \pt$,
where the map $\pt \to \Affl$ is given by the origin.
Moreover,
we denote the base changed morphisms by
$f_\eta \colon X_\eta \to \Gm$, and
$f_\sigma \colon X_\sigma \to \pt$,
as in the following diagram:
\begin{equation*}
\begin{tikzcd}
X_{\sigma} \arrow[r, hook, "i"] \arrow[d, "f_\sigma" ']
& X  \arrow[d, "f"]
& X_{\eta} \arrow[l, hook', "j" '] \arrow[d, "f_{\eta}"]
\\
\pt \arrow[r, hook]
& \Affl
& \Gm \arrow[l, hook']
\end{tikzcd}
\end{equation*}
\end{notation}


\subsection{%
	Singularities and 
	semi\-stable reduction%
}

\begin{defn}
	\label{definition:sss}
Let $f\colon X \to \Affl$ be a flat morphism, with $X$ regular.
\begin{enumerate}
\item
We say that $f$ is \emph{semistable} if the special fibre $X_\sigma$ is a simple normal crossing divisor, in particular, reduced.
\item
We say that $f$ is \emph{special semistable} if it is semistable,
and in addition the special fibre $X_\sigma$ is given as the union of two smooth divisors.
\end{enumerate}
\end{defn}

\begin{defn}\label{definition:ssr}
Let $f\colon X \to \Affl$ be a morphism.
Let $r \colon \Affl \to \Affl$ be
the $r$th power map,
for $r$ a natural number.
We also use $r$ to denote
the base change morphism
$X_r := X\times_{\Affl,r}\Affl \to X$
and $f_r \colon X_r \to \Affl$
to denote the base change of $X \to \Affl$.
We say that $f$
\emph{admits (special) semistable reduction}
if
there is a natural number $r$ 
and a proper map
$\pi \colon Y \to X_r $
such that
$g := f_r \circ \pi \colon Y \to \Affl$
is a (special) semistable morphism.
\end{defn}

\begin{defn} \label{definition:goodmorphism}
Let $f \colon X \to \Affl$
be a morphism that admits special semistable reduction
\begin{equation*}
g \colon Y \xrightarrow{\pi} X_r \xrightarrow{f_r} \Affl,
\end{equation*}
as in Def.~\ref{definition:ssr}.
Assume furthermore that
the special fibre $X_\sigma$
has one isolated singular point $o \in X_\sigma$,
and that $(\pi \circ r)^{-1}(o) =: D$
is a proper smooth branch in
$Y_\sigma$.
In this case,
we say that $f$ is a \emph{family with a good singularity}.
\end{defn}

This definition suffices for the semistable reduction argument
in \S~\ref{sec:ss_reduction}.
Its conditions are satisfied in the case we are interested in,
namely that of a homogeneous singularity,
and more generally of a quasi-homogeneous singularity.

\begin{defn}\label{def:homogeneoussingularity}
Let $f\colon X \to \Affl = \Spec k[t]$ be
a flat, quasi-projective morphism,
let
$F \in k[T_0,\dotsc,T_n]$
be a homogeneous polynomial of degree $r$,
with $(r,\operatorname{char} k)=1$.
We say that a
singular
point of the special fibre, $o \in X_\sigma$, is
\emph{a homogeneous singularity defined by $F$,}
if
the hypersurface $V(F) \subset \P_k^n$ is smooth,
and we have
\[ f^*(t) = F(x_0, \dotsc, x_n) \]
in the local ring $\OO_{X,p}$ modulo $m^{r+1}$,
where $(x_0, \dotsc, x_n)$ is a regular sequence
generating the maximal ideal $m \subset \OO_{X,p}$.
\end{defn}

\begin{rmk}
In the case $r=2$,
this is an ``ordinary quadratic singularity''
as defined in~\cite[Exp.~XV, Def.~1.2.1]{SGA7}.
\end{rmk}

Homogeneous singularities are good,
in the sense of Def.~\ref{definition:goodmorphism}:

\begin{prop}\label{proposition:homogeneousssred}
Let $f\colon X \to \Affl$ be
a flat, quasi-projective morphism.
Suppose that $X_\sigma$ has
a single homogeneous singularity $o$
defined by $F \in k[T_0, \dotsc, T_n]$ of degree $r$.
\begin{enumerate}
\item
We have that $f$ is a family with a good singularity.

\item
Special semistable reduction is given by $r$th power base change,
followed by blowup at $o$ and normalization.

\item
Zariski locally around $o$,
the special semistable morphism
$g \colon Y \to (\Affl)_r$
has special fibre
$Y_\sigma = D_1 \cup D_2$,
where
$D_1 \simeq V_{\P_k^{n+1}}(F-T_{n+1}^r)$,
$C \simeq V_{\P_k^n}(F)$,
and $D_2$ is the strict transform of $X_\sigma$.
\end{enumerate}
\end{prop}

\begin{proof}
Under the assumption that
$V_{\P^{n+1}}(F-T_{n+1}^r)$
is smooth, this is~\cite[Prop.~2.4]{illusiePL}.
For the general case,
see~\cite[Thm.~4.3]{Az}.
The statement in~[loc. cit.]
is formulated for a family over a discrete valuation ring,
but the same proof works,
as we can restrict $f$ to $\Spec k[t]_{(t)}$.
\end{proof}

We now turn to the larger class of \emph{quasi-homogeneous singularities}, where our definition requires some technical restrictions, but they are still \emph{good} in the sense of Def.~\ref{definition:goodmorphism}.

\begin{defn}\label{def:qhomogeneoussingularity}
	Let
	$\underline{a}=(a_1, \dotsc, a_n)$
	be a vector of natural numbers and let $\P(\underline{a})$ be the weighted homogeneous space $\operatorname{Proj} k[T_0, \dotsc, T_n]$, with the ring graded by letting $T_i$ have degree $a_i$. Let $F \in k[T_0, \dotsc, T_n]$ be an $\underline{a}$-weighted homogeneous polynomial of degree $r$. It defines a hypersurface $V(F)$ in the weighted homogeneous space $\P(\underline{a})$. We require that the weights $a_i$ are pairwise relatively prime, each $a_i$ divides $r$, and $r$ is prime to the exponential characteristic of $k$.
	Let
	$G(y_0, \dotsc, y_n) := F(y_0^{a_0}, \dotsc, y_n^{a_n})$
	and assume that both $V(G) \subset \P_k^n$, and $V(F) \subset \P_k(\underline{a})$ are smooth. Furthermore, letting $v_i\in   \P_{k(p)}(a)$ be the point with the $i$-th homogeneous coordinate 1, and all other coordinates 0,  we require that $F(v_i)\neq0$ if $a_i>1$ (this last condition is superfluous in the case $n > 1$, see \cite[Rmk.~5.4]{Az}).
	
	Let $f\colon X \to \Spec k[t]$ be a flat, quasi-projective morphism.	We say that a singular point $o \in X_\sigma$ is \emph{an $\underline{a}$-weighted homogeneous singularity defined by $F$}, if in the local ring $\OO_{X,p}$, we can write
	\[ f^*(t) = F(x_0, \dotsc, x_n) + h,\]
	where $\underline{x}=(x_0, \dotsc, x_n)$ is a regular sequence generating the maximal ideal $m \subset \OO_{X,p}$, and $h \in m \cdot m_{\underline{a}}^{(r)}$, where $m_{\underline{a}}^{(r)}$ is the ideal generated by monomials in $\underline{x}$ of $\underline{a}$-weighted degree $r$. For more details see~\cite[Def.~5.1, Def.~5.3]{Az}.
\end{defn}

\begin{prop}\label{proposition:qhomogeneousssred}
Let $f\colon X \to \Aa^1_k$ be
a flat, quasi-projective morphism.
Suppose that $X_\sigma$ has a single quasi-homogeneous singularity $o$
defined by $F \in k[T_0, \dotsc, T_n]$ of degree $r$.
Then $f$ is a family with good singularity,
and it admits special semistable reduction
$g \colon Y \to \Affl$,
where $g = f_r \circ \pi$,
$f_r \colon X_r \to \Aff^1$
is $r$th power base change of $f$,
and $\pi \colon Y \to X_r$ is proper.
We have $Y_\sigma = D_1 \cup D_2$, with
$D_1 \simeq V_{\P_k(\underline{a},1)}(F-T_{n+1}^r)$,
and
$C \simeq V_{\P_k(\underline{a})}(F)$.
\end{prop}	

\begin{proof}
This follows from \cite[Thm.~5.7]{Az}.
While the construction of $g$ is more involved than in
Prop.~\ref{proposition:homogeneousssred}
(replacing the blowup by a certain construction of a weighted blowup),
it still satisfies the statement of the proposition.
\end{proof}

\begin{notation}\label{notation:ss}
Let $f\colon X \to \Affl$ be a semistable morphism
and $X_\sigma = \sum_i D_i$
where $D_i$ are smooth, irreducible divisors.
Let $C := \bigcap_i D_i$ denote the intersection,
and let $D_i^\circ := D_i \setminus \bigcup_{j\neq i}D_j$.
We use the following notation for some natural inclusion maps:
\begin{equation*}\begin{tikzcd}
C \ar[dr, bend right=16, "c_i" '] \ar[r, "c_{0,i}"] & D_i \ar[d, "u_i"] & D_i^\circ \ar[l, "v_i" '] \\
& X_\sigma
\end{tikzcd}\end{equation*}
\end{notation}


\subsection{Motives}

Let for a variety $X$ over $k$,
let $\DA(X)$ denote the category
$\DA^\rom{et}(X,\mathbb{Q})$
of étale motivic sheaves over $X$ with rational coefficients,
see for example~\cite[\S~3]{ayoub_etale}.
These are the motivic categories we work with
throughout the paper.

Recall that these categories allow for a six-functor formalism.
We use this structure throughout the paper.

\begin{notation}
Let $f \colon X \to S$ be a morphism.
We use the following notation for the cohomological motive
and the motive with compact support of $X$ over $S$, respectively:
\[
h_S(X) := f_* f^* \one_S,
\quad h^\rom{c}_S(X) := f_! f^* \one_S
\] 
(as objects in $\DA(S)$).
\end{notation}

This basic vanishing result is used numerous times in the paper:

\begin{prop}[Voevodsky, see~{\cite[Hyp.~3.6.47]{ayoub_thesis_2}}]
\label{prop:neg_twist_vanishing}
Let $i,j,k,n$ be integers and suppose $n>0$.
Then, for $X$ smooth, we have
\begin{equation*}
\Hom_{\DA(X)}(\bm1(i)[j], \bm1(i-n)[k]) = 0.
\end{equation*}
\qed
\end{prop}


\subsection{The Kummer motive}

Key to Ayoub's definition of the motivic monodromy operator,
and to the proof of our main theorem,
is the Kummer motive.
Before we recall its definition,
we remind the reader about the standard computation
of the motive of $\Gm$.

\begin{rmk}[The motive of $\Gm$.]\label{rmk:motiveGm}
Consider the usual, complementary embeddings
$j\colon \Gm \hookrightarrow \Affl$,
and $i = 0 \colon \pt \hookrightarrow \Affl$.
Gluing along these (\cite[Lemme~1.4.6]{ayoub_thesis_1})
gives a fibre sequence
\begin{equation*}
i_!i^!\bm1 \to \bm1 \to j_*j^*\bm1
.
\end{equation*}
As $i$ is proper, $i_! = i_*$,
and by purity (\cite[\S~1.6.3]{ayoub_thesis_1}),
$i^!\bm1 = \bm1(-1)[-2]$.
Our sequence becomes
\begin{equation*}
i_*\bm1(-1)[-2] \to \bm1 \to j_*\bm1
.
\end{equation*}
Upon shifting the sequence,
applying $p_*$,
and using $h(\Affl) = h(\pt) = \bm1$,
we get a fibre sequence
\begin{equation*}
\bm1 \to h(\Gm) \to \bm1(-1)[-1]
.
\end{equation*}
The canonical point
$1 \colon \pt \hookrightarrow \Gm$
splits the leftmost map,
yielding
\begin{equation*}
h(\Gm) = \bm1 \oplus \bm1(-1)[-1]
.
\end{equation*} 
\end{rmk}

With this,
we are ready to define to recall
Ayoub's definition of the Kummer motive.
We use the notation
\begin{equation*}
	\Gm \xrightarrow{\Delta} \Gm\times\Gm \xrightarrow{\rom{pr}_1} \Gm,
\end{equation*}
for the diagonal and the projection onto the first factor, respectively.
We also need the horizontal section
$\id\times1 \colon \Gm \to \Gm\times\Gm$.
As in Rmk.~\ref{rmk:motiveGm},
the point $1\in\Gm$ induces
an isomorphism
\begin{equation*}
\rom{pr}_{1,*}\rom{pr}_1^*\bm1 \simeq \bm1 \oplus \bm1(-1)[-1]
.
\end{equation*}
In particular there is a map
$\bm1(-1)[-1] \to \rom{pr}_{1,*}\rom{pr}_1^*\bm1$,
which we use in the following definition.

\begin{defn}[{\cite[Def.~3.6.22, Lemme~3.6.28]{ayoub_thesis_2}}]\
\label{defn:Kummer}
\begin{enumerate}
\item
The composition
\begin{equation*}
\bm1(-1)[-1]
\to \rom{pr}_{1,*}\rom{pr}_1^*\bm1
\to \rom{pr}_{1,*}\Delta_*\Delta^*\rom{pr}_1^*\bm1
\simeq \bm1
,
\end{equation*}
is the \emph{Kummer map}.
We denote it by $e_\rom{K}$.
\item
The \emph{Kummer motive}, denoted $\clg{K} \in \DA(\Gm)$,
is the cofibre of the Kummer map $e_\rom{K}$:
\begin{equation*}
\bm1(-1)[-1]
\xrightarrow{e_\rom{K}} \bm1
\xrightarrow{c_\rom{K}} \clg{K}
.
\end{equation*}
\end{enumerate}
\end{defn}


\subsection{Nearby cycles}

Let $f \colon X \to \Affl$ be a quasi-projective morphism
and recall the setup in Notation~\ref{notation:fam_over_A1}.
The motivic nearby cycles functor
\[
\Psi_f \colon \DA(X_{\eta}) \rightarrow \DA(X_{\sigma})
,
\]
is constructed
(in several ways)
in~\cite[\S~3]{ayoub_thesis_2}
and~\cite[\S\S~10--11]{ayoub_etale}.

We briefly recall one of his constructions.

\begin{const}[{%
	\cite[\S~3.6]{ayoub_thesis_2},
	\cite[\S~11]{ayoub_etale}%
}]
\label{const:unip_nearby_cycles}
Recall the Kummer motive $\clg{K}$
and the canonical map
$c_\rom{K} \colon \bm1 \to \clg{K}$
in $\DA(\Gm)$.
We define the logarithm motive
\begin{equation*}
\Log^\vee := \operatorname{Free}_{/\bm1}(\clg{K})
\end{equation*}
as the free algebra on $\clg{K}$ over $\bm1$.\footnote{%
	Ayoub uses the notation $\Sym^\infty\clg{K}$,
	but we are going with the notation
	in~\cite{motivic_monodromy}
	here.
	See [ibid., Rmk.~3.7].%
}
The motivic nearby cycles functor is then defined as
\begin{align*}
\Psi_f \colon \DA(X_\eta)
& \to \DA(X_\sigma)
\\
M
& \mapsto i^*j_*(M \otimes f_\eta^*\Log^\vee)
.
\end{align*}
\end{const}

\begin{rmk}\label{rmk:upsilon}
By~$\Psi_f$,
we denote throughout the paper Ayoub's
\emph{unipotent} nearby cycles functor,
which is denoted by~$\Upsilon_f$ in~\cite{ayoub_thesis_2}.
This is the version of nearby cycles for which
the monodromy operator is defined.
Ayoub also constructs the total nearby cycles functor,
which he denotes by~$\Psi_f$.
Both functors agree for a semistable family
(this follows from \cite[Thm.~3.3.44]{ayoub_thesis_2}).

To complicate things further,
the construction above is actually
of the \emph{logarithmic specialization system},
which Ayoub denotes by $\operatorname{log}_f$.
However,
he shows,
in~\cite[Thm.~3.6.44]{ayoub_thesis_2}
and~\cite[Thm.~11.14]{ayoub_etale},
that $\Upsilon_f \simeq \operatorname{log}_f$.
\end{rmk}

The following properties of
motivic nearby cycles
are foundational.

\begin{prop}[{\cite[Def.~3.1.1, Prop.~3.2.9]{ayoub_thesis_2}}]
\label{prop:nearby_cycles_pushf_pullb}
For every morphism
$ g \colon Y \rightarrow X $
of quasi-projective $\Affl$-schemes,
there are natural transformations
\[
\alpha_g
\colon g_{\sigma}^* \circ \Psi_f
\to \Psi_{f \circ g} \circ g_{\eta}^*
,
\quad
\beta_g
\colon \Psi_f \circ g_{\eta *}
\to g_{\sigma *} \circ \Psi_{f \circ g}
\]
such that:
\begin{enumerate}
\item
If $g$ is smooth, then $\alpha_g$ is an isomorphism.
\item
If $g$ is projective, then $\beta_g$ is an isomorphism.
\qed
\end{enumerate}
\end{prop}

These natural transformations $\alpha$ and $\beta$ satisfy various compatibility conditions,
and there are exceptional variants $\mu$ and $\nu$.
For details, see \cite[\S\S~3.1.1--3.1.2]{ayoub_thesis_2}.

\begin{prop}[{\cite[Prop.~10.1]{ayoub_etale}}]\label{prop:psione}
	We have
		\[
	\Psi_{\id} \one \simeq \one .
	\]
	\qed
\end{prop}

The following result is our main tool for computing nearby cycles.

\begin{prop}[{\cite[Thm.~3.3.44]{ayoub_thesis_2}}]
\label{prop:computability}
Let
$f \colon X \to \Affl$
be a
quasi-pro\-ject\-ive morphism.
Suppose that $f$ is semistable
and recall Notation~\ref{notation:ss}.
Then,
for any of branch $D_i$,
the unit map $\id \rightarrow v_{i,*} v_i^*$
induces a natural isomorphism
\begin{equation}\label{eq:computability1}
u_i^* \Psi_f f_\eta^*
\xrightarrow{\sim} v_{i,*} v_i^* u_i^* \Psi_f f_\eta^* 
.
\end{equation} 
Similarly,
\begin{equation}\label{eq:computability2}
v_{i,!} v_i^! u_i^! \Psi_f f_\eta^!
\xrightarrow{\sim} u_i^! \Psi_f f_\eta^!
.
\end{equation} 
\qed
\end{prop}

\begin{cor}\label{cor:computability}
With the same notation as above, we have
\[
u_i^* \Psi_f \one \simeq v_{i,*} \one
,
\quad
u_i^! \Psi_f \one \simeq v_{i,!} \one
.
\]
\end{cor}

\begin{proof}
Let $w\colon D_i^\circ\to\pt$ be the structure morphism.
Applying the compatibility with smooth pullback
(Prop.~\ref{prop:nearby_cycles_pushf_pullb}),
first to the open immersion
$X \setminus \cup_{j\neq i} D_j\hookrightarrow X$,
then to the smooth morphism
$X \setminus \cup_{j\neq i} D_j \to \Affl$,
and using Prop.~\ref{prop:psione},
we get
\[
v_i^* u_i^* \Psi_f \one = w_i^* \Psi_{\id} \one = \one
.
\]
Combining this with Prop.~\ref{prop:computability},
we get the first equivalence.
The second equivalence is handled similarly,
as we have
\[
v_i^! u_i^! \Psi_f f_\eta^! \one
\simeq v_i^! u_i^!  f_\sigma^! \Psi_{\id} \one
\simeq w_i^! \Psi_{\id} \one
\simeq w_i^! \one
\] 
for the exceptional pullback functor,
and $f_\eta$ and $w_i$ are both smooth and of the same dimension.
\end{proof}


\subsection{Monodromy and variation}

Consider a quasi-projective family $f \colon X \to \Affl$
and recall Notation~\ref{notation:fam_over_A1}.
Let
$\chi_f := i^*j_*$
denote the canonical specialization functor
(\cite[Ex.~3.1.4]{ayoub_thesis_2}).
By~\cite[Thm.~11.16]{ayoub_etale},
there is a canonical cofibre sequence
\[
\chi_f \to \Psi_f \xrightarrow{N} \Psi_f(-1)
\]
and the map $N \colon \Psi_f \to \Psi_f(-1)$
is called the \emph{monodromy operator}.

There is no monodromy on $\Psi_f\bm1$ when $f$ is smooth:
in this case,
$\Psi_f\bm1 \simeq \bm1$
(by Prop.~\ref{prop:nearby_cycles_pushf_pullb}
and Prop.~\ref{prop:psione}),
and there is no non-zero map $\bm1\to\bm1(-1)$
(Prop.~\ref{prop:neg_twist_vanishing}).
This leads to the monodromy operator being
supported on the singular locus of $f$.
More precisely,
let
$i \colon Z \hookrightarrow X_\sigma$
be a closed subscheme containing the singular locus of $X_\sigma$,
and denote its open complement by $j\colon U \hookrightarrow X_\sigma$.
By gluing~\cite[Lemme~1.4.6]{ayoub_thesis_1},
and the vanishing of $j^*N$, we get
\begin{equation*}\begin{tikzcd}
j_!j^*\Psi_f\bm1 \ar[d, "0"] \ar[r]
  & \Psi_f \ar[d, "N"] \ar[r]
  & i_*i^*\Psi_f\bm1 \ar[dl, dotted] \ar[d, "i_*i^*N"] \\
j_!j^*\Psi_f\bm1(-1) \ar[r] & \Psi_f(-1) \ar[r] & i_*i^*\Psi_f\bm1(-1).
\end{tikzcd}\end{equation*}
Thus,
the monodromy operator factors through the unit map
$\Psi_f\bm1 \to i_*i^*\Psi_f\bm1$.
Using the dual localization sequence,
we get in the same way that
it factors through the natural map
$i_*i^!\Psi_f\bm1(-1) \to \Psi_f\bm1(-1)$.
In summary:

\begin{prop}\label{prop:varmap}
There is a unique map
$\varmap\colon i^*\Psi_f\bm1 \to i^!\Psi_f\bm1(-1)$,
which we call the \emph{variation map} (over $Z$),
that factors the monodromy operator $N$ as
\begin{equation*}
\Psi_f\bm1
\xrightarrow{\rom{un}} i_*i^*\Psi_f\bm1
\xrightarrow{i_*\varmap} i_*i^!\Psi_f\bm1(-1)
\xrightarrow{\rom{coun}} \Psi_f\bm1(-1).
\end{equation*}
\qed
\end{prop}

There are times when we need to know
a bit more about the construction of the monodromy operator.

\begin{const}[{\cite[\S~11]{ayoub_etale}}]\label{const:monodromy}
Recall Construction~\ref{const:unip_nearby_cycles}.
The canonical map
$\clg{K} \to \bm1(-1)$
induces a map
$\Log^\vee \to \Log^\vee(-1)$.
This induces the monodromy operator by functoriality
in the formula
$\Psi_f = \chi_f(\blank \otimes f_\eta^*\Log^\vee)$.
\end{const}


\section{%
	First reduction:
	to a semistable family%
}\label{sec:ss_reduction}

In this section we describe how to reduce
computing the monodromy from
the case of a homogeneous or quasi-homogeneous singularity to
that of a special semistable family.

\subsection{The variation map and semistable reduction}

Here we wish to express $\var_f$ in terms of $\var_g$ for a semistable reduction construction. In order to do that we have to keep track of how base change, proper morphisms, and smooth morphisms affect the variation map. First we inspect how base change affects the unipotent nearby cycles and the monodromy.

\begin{prop}[*]\label{prop:basechange}
	Let $f\colon X \to \Affl$ be a morphism
	and let $N_f \colon \Psi_f \one \to \Psi_f \one (-1)$
	be the corresponding monodromy.
	Let $r\colon \Aff^1 \to \Aff^1$ be the $r$th power map
	and let  $f'\colon X' \to \Aff^1 $ be the base change morphism.
	Then we have a canonical isomorphism $\Psi_{f} \simeq \Psi_{f'} \circ r_\eta^*$, under which we have
	\[
	N_{f'} = r \cdot N_f
	.
	\]
\end{prop}

\begin{proof}
	First we analyse the base change influence on the Kummer map $e_K$, then we deduce it for the Kummer motive $\mathcal{K}$, then for $\mathcal{L}og^\vee$, and therefore for $\Psi$.
	Recall the definition of the Kummer map. We have the composition
	\begin{align*}
		\Gm \xrightarrow{\Delta} \Gm\times\Gm \xrightarrow{\rom{p}} \Gm,
	\end{align*}
	where $\rom{p}$ is the projection on the first coordinate. We write the Kummer map for $f'$ as $e_K' \colon \one_{\Gm}(-1)[-1] \to \one_{\Gm}$:
	\begin{equation*}
		\bm1(-1)[-1]
		\to \rom{p}_{*}\rom{p}^*\bm1
		\to \rom{p}_{*}\Delta_*\Delta^*\rom{p}^*\bm1
		\simeq \bm1.
	\end{equation*}
	
	We compare this with $r^* e_K$, which can be described by the composition
	\begin{align*}
		\bm1(-1)[-1]
		\to \rom{p}_{*}\rom{p}^*\bm1
		\to \rom{p}_{*}\tau_* \tau^* \rom{p}^*\bm1\simeq  \rom{p}_{*}\rom{p}^*\bm1
		\to \rom{p}_{*}\tau_*\Delta_*\Delta^*\tau^*\rom{p}^*\bm1
		\simeq \bm1,
	\end{align*}
	where $\tau \colon \Gm \times \Gm \to \Gm \times \Gm$ is given by $(x,y) \mapsto (x, y^r)$. 
	The two maps differ then by the middle endomorphism $\rom{p}_{*}\rom{p}^*\bm1_{\Gm}
	\to \rom{p}_{*}\tau_* \tau^* \rom{p}^*\bm1_{\Gm} \simeq  \rom{p}_{*}\rom{p}^*\bm1_{\Gm}$ of $\rom{p}_{*}\rom{p}^*\bm1_{\Gm} \simeq \one_{\Gm} \oplus \one_{\Gm}(-1)$. This automorphism respects the decomposition, and acts by $\times r$ on the first component and as $\rom{id}$ on the second component. Therefore, we have the commutative diagram
	\[\begin{tikzcd}
		r^*\one(-1)[-1] = \one(-1)[-1] \arrow[d, "\times r" '] \arrow[r, "r^* e_\mathcal{K}"]
		& r^* \one = \one	\arrow[d, "\rom{id}"] \\
		\one(-1)[-1] \arrow[r, "e'_\mathcal{K}"]
		& \one .
	\end{tikzcd}\]
	Taking the cofibre and shifting the sequence,
	we get an isomorphism of fibre sequences
	\[\begin{tikzcd}
		\one\arrow[d, "\rom{id}"] \ar[r] &
		r^*\mathcal{K} \arrow[r] \ar[d, "\sim" {rotate=90, anchor=north}] &
		\one(-1) \ar[d, "\times r" ] 
		\\
		\one \ar[r] &
		\mathcal{K} \ar[r] &
		\one(-1)
		.
	\end{tikzcd}\]
	We claim that the above diagram induces an isomorphism of fibre sequences
	\[\begin{tikzcd}
		\one	\arrow[d, "\rom{id}"] \ar[r]&
		r^*\Log^\vee  \arrow[r, "r^*N"] \ar[d, "\sim" {rotate=90, anchor=north}]&
		r^*\Log^\vee(-1) \ar[d, " \times r" ]
		\\
		\one \ar[r] &
		\Log^\vee  \ar[r, "N"] &
		\Log^\vee(-1)
		.
	\end{tikzcd}\]
	Then by definition of $\Psi$ and the middle isomorphism,
	we have an isomorphism $\Psi_{f} \simeq r_{\sigma, *} \circ \Psi_{f'} \circ r_\eta^*$, and moreover, $N_{f'} \simeq r N_f$.
	A more detailed proof will appear in a later version of this paper.
\end{proof}

We now show that on a good singularity,
the motive of nearby cycles is actually computable.

\begin{prop}\label{proposition:compute nearby cycles}
Let $f\colon X \to \Aff^1$ be a morphism with a good singularity
(Def.~\ref{definition:goodmorphism}).
Let $\pi \colon Y \to X_r$ be a proper morphism such that
$g = f_r \circ \pi$
is special semistable according to the definition,
with $D= g^{-1}(o)$ be a proper branch of $Y$.
Let the restrictions of $\pi$ and the inclusions
be denoted according to the following diagram:
\[\begin{tikzcd}[row sep=large ]
D \arrow[r, "u", hook] \arrow[d, "\pi'"] & Y_\sigma  \arrow[d, "\pi_\sigma"]  \\
o \arrow[r, hook, "o"] & X_\sigma .
\end{tikzcd}\]
Let $v\colon D^\circ := D \setminus C \to D$.
We then have the equivalences
\begin{equation}\label{eq:psissreduction}
\Psi_f \bm1_{X_\eta} \simeq   \pi_{\sigma *} \Psi_g \bm1_{Y_\eta} ,
\end{equation}
\begin{equation}\label{eq:computenearbycycles1}
\Ay^* \colon o^* \Psi_f \bm1 \simeq \pi'_* v_* \bm1_{D^\circ} = h(D^\circ),
\end{equation}
and
\begin{equation}\label{eq:computenearbycycles2}
\Ay^! \colon o^! \Psi_f \bm1 \simeq \pi'_* v_! \bm1_{D^\circ} = h_\rom{c}(D^\circ).
\end{equation} 
\end{prop}

\begin{rmk}
Prop.~\ref{proposition:compute nearby cycles} is a motivic version
of~\cite[Thm.~2.6(b)]{illusiePL}.
\end{rmk}

\begin{proof}
	First, by \ref{prop:basechange}, we have $ \Psi_f \one \simeq \Psi_{f_r} \one$, therefore we can assume that $r=1$ and $g= f \circ \pi$. Now by Prop.~\ref{prop:nearby_cycles_pushf_pullb}(2), $\Psi_f \pi_{\eta *}  \simeq \pi_{\sigma*}\Psi_g $, and so, using the fact that $\pi_\eta $ is an isomorphism, we have
	\[ \Psi_f \bm1_{X_\eta} \simeq \Psi_f \pi_{\eta *} \pi_{\eta}^* \bm1_{X_\eta} \simeq \Psi_f \pi_{\eta *} \bm1_{Y_\eta} \simeq  \pi_{\sigma *} \Psi_g \bm1_{Y_\eta} .\]
	Now, by proper base change, $  o^* \pi_{\sigma*}  \simeq   \pi'_* u^*  $,
	by Cor.~\ref{cor:computability}, $u^* \Psi_g \bm1_{Y_\eta} \simeq v_* \bm1_{D^\circ}$, and so
	\[ o^* \Psi_f \bm1_{X_\eta} \simeq o^*  \pi_{\sigma *} \Psi_g \bm1_{Y_\eta} \simeq  \pi'_* u^* \Psi_g \bm1_{Y_\eta} \simeq \pi'_* v_* \bm1_{D^\circ}.\]
	The second isomorphism is obtained similarly by using
	the compact support version of proper base change instead
	($i^! \pi_{\sigma!} \simeq \pi'_! u^! $)
	and the fact that, since $\pi_\sigma$ is proper, $\pi_{\sigma!} \simeq \pi_{\sigma*} $ and $\pi'_{!} \simeq \pi'_{*} $.
\end{proof}

\begin{prop}\label{prop:varcomparison}
Let $f \colon X \to \Aa^1$ with singular locus
$o \colon S \hookrightarrow X_\sigma$,
let	$g = f_r \circ \pi$ with $f_r$ being $r$-base change of $f$,
and let $\pi \colon Y \to X_r$ a proper map.
Let $c \colon C \to Y_\sigma$ be the singular locus of $Y$. We then have the following commutative diagram:
\[\begin{tikzcd}
&o^*\Psi_f\one  \arrow[d, "Ex_*^*" ']  \arrow[rr, "r \cdot \var_f"] & & o^!\Psi_f\one(-1) \\
& \pi_{C *}c^* \Psi_g \one \arrow[rr, "\pi_{C *}c_*\var_g"]
&	&[2em] \pi_{C *}c^! \Psi_g\one(-1) \arrow[u, "Ex_*^!" '] &
\end{tikzcd}\]
where the vertical maps are given by the exchange morphisms
$Ex_*^* \colon o^*\pi_{\sigma *} \to \pi_{C *} c^* $,
and $Ex_*^! \colon \pi_{C *} c^! \to o^! \pi_{\sigma *}$
relative the commutative square
\[\begin{tikzcd}
C  \ar[d, "\pi_C"'] \ar[r, "c", hook]
& Y_\sigma \ar[d, "\pi_\sigma"] \\
S \ar[r, "o", hook] & X_\sigma,
\end{tikzcd}\] 
and modulo the identification~\eqref{eq:psissreduction}.
\end{prop}	

\begin{proof}
We know already how the base-change affects the monodromy so we can assume $r=1$ and multiply the result by $r$.
Writing the decomposition of Prop.~\ref{prop:varmap} for the monodromy of $f$ and $g$, pushing forward $N_g$ by $\pi_{\sigma *} $, and using the equivalence \eqref{eq:psissreduction} on the terms of $N_f$, we have the diagram
\[\begin{tikzcd}
\pi_{\sigma *}\Psi_g\one \arrow[r, "\eta_o"] \arrow[d, equals ]
&  o_*o^*\pi_{\sigma *}\Psi_g\one \arrow[ r, "r \cdot o_* \var_f"] \arrow[d, "Ex_*^*" ']
&[2em] o_*o^!\pi_{\sigma *}\Psi_g\one(-1) \arrow[r, "\epsilon_o"]
& \pi_{\sigma *}\Psi_g\one(-1) \arrow[d, equals ]
\\
\pi_{\sigma *}\Psi_g\one \arrow[r, "\eta_c"] 
& \pi_{\sigma *}c_*c^*\Psi_g\one \arrow[r, " \pi_{\sigma *}c_* \var_g"]
&[2em] \pi_{\sigma *}c_*c^!\Psi_g\one(-1)
	\arrow[r, "\epsilon_C"]
	\arrow[u, "Ex_*^!" ']
& \pi_{\sigma *}\Psi_g\one(-1) .
\end{tikzcd}\]
The outer rectangle commutes from the naturality of monodromy,
so $\pi_{\sigma *} N_g = N_f $.
The leftmost square commutes,
since by definition $Ex_*^*$ is given by
precomposing with $\eta_c$ and composing with $\epsilon_o$.
Similarly, the rightmost square commutes
Then, the middle square commutes by
all of the above and
the definition of $\var$.
Applying $o^*$ to this central square
and using that it is fully faithful,
we get the desired resulting square.
\end{proof}
	
In order to have freedom to choose a semistable model, we also similarly show how $\var$ behaves along smooth pullbacks.  

\begin{prop}\label{prop:varsmoothcomparison}
Let
$g \colon Y \to \Aa^1$
be a morphism with singular locus
$c \colon C \hookrightarrow Y_\sigma$,
let $g' \colon Y' \to \Aa^1$ be with singular locus
$c' \colon C' \to Y'_\sigma$,
and
let $\rho \colon Y \to Y'$
be a smooth map such that
$g = g' \circ \rho $,
and $C = \rho^{-1} (C')$.
 Then,
 \[
 \rho_C^* \var_{g'} \simeq \var_g
 .
 \]
More precisely, the diagram
\[\begin{tikzcd}
\rho_C^* c'^*\Psi_{g'}\one
	\arrow[d, "Ex_*^*" ']
	\arrow[rr, " \rho_C^* \var_{g'}"]
&& \rho_C^* c'^!\Psi_{g'}\one(-1)
	\arrow[d, "Ex^{!,*} \circ Ex_*^*"]
\\
c^* \Psi_g \one
	\arrow[rr, "\var_g"]
&&[2em] c^! \Psi_g\one(-1)
\end{tikzcd}\]
commutes.
The vertical maps are isomorphisms given by the exchange isomorphisms
$Ex_*^* \colon \rho_\sigma^* c'_{*} \to c_{*} \rho_C^* $,
and
$Ex^{!,*} \colon \rho_C^*  c'^! \to c^! \rho_\sigma^*$
relative the smooth pullback square
\[\begin{tikzcd}
C  \ar[d, "\rho_C"'] \ar[r, "c", hook]
& Y_\sigma \ar[d, "\rho_\sigma"] \\
C' \ar[r, "c'", hook] & Y_\sigma',
\end{tikzcd}\] 
and modulo the equivalence
$\rho_\sigma^* \Psi_{g'} \one \simeq \Psi_{g} \one $
given by the map $\beta_\rho$
from Prop.~\ref{prop:nearby_cycles_pushf_pullb}.
\end{prop}	
	
\begin{proof}

		Writing the decomposition of Prop.~\ref{prop:varmap} for the monodromy of $f$ and $g$, pushing forward $N_g$ by $\pi_{\sigma *} $, and using the equivalence \eqref{eq:psissreduction} on the terms of $N_f$, we have the diagram
		\[\begin{tikzcd}
			\rho_{\sigma}^* \Psi_{g'} \one \arrow[r, "\eta_{c'}"] \arrow[d, "\simeq"' ]
			& \rho_{\sigma}^* c'_*c'^*\Psi_{g'}\one \arrow[ r, "\rho_{\sigma}^* c'_* \var_{g'}"] \arrow[d, "\simeq", "Ex_*^*" '] &  [2em] \rho_{\sigma}^*c'_*c'^!\Psi_{g'}\one(-1)  \arrow[r, "\epsilon_{c'}"] \arrow[d, "Ex^{!,*} \circ Ex_*^*", "\simeq" ']  & \rho_{\sigma}^* \Psi_{g'}\one(-1) \arrow[d, "\simeq" ] \\
			\Psi_{g}\one \arrow[r, "\eta_c"] 
			&c_*c^*\Psi_g\one \arrow[r, " c_* \var_g"]  &  [2em] c_*c^!\Psi_g\one(-1)  \arrow[r, "\epsilon_C"] & \Psi_g\one(-1) .
		\end{tikzcd}\]
		The outer rectangle commutes from the naturality of monodromy, as $\rho_{\sigma}^* N_{g'} = N_g $. The leftmost square commutes since by definition $Ex_*^*$ is given by precomposing with $\eta_c$ and composing with $\epsilon_{c'}$. The rightmost square also commutes, since $Ex^{!,*}$ is given by applying $\eta^!_C$, $(Ex_*^*)^{-1}$ and then $\epsilon^!_{c'}$ (see \cite[Prop. 1.4.15]{ayoub_thesis_1}). The vertical arrows are all isomorphisms since $\rho_{\sigma}$ is smooth and by smooth base change properties (for $Ex^{!,*}$ see \cite[Prop. 1.4.17]{ayoub_thesis_1}). The middle square commutes then by all  of the above and the definition of $\var$.
		Applying $c^*$ to this central square and using the fact that, being an immersion, it is fully faithful (and similarly $c'$), we get the desired result.
\end{proof}
	
\subsection{The monodromy of a special semistable family}

\begin{prop}\label{prop:motivetriangle}
Let
$ C \xhookrightarrow{i} D \xhookleftarrow{j} D^\circ := D\setminus C $
be closed and open complementing immersions into a scheme $D$,
where $C$ is a subscheme of codimension $1$ in $D$.
We have the following fibre sequences in $\DA(D)$:
\begin{equation}\label{eq:motivetrianglea}
h_D(D) \to h_D(D^\circ) \xrightarrow{\alpha} h_D(C)(-1)[-1]
\end{equation}	
\begin{equation}\label{eq:motivetriangleb}
h_D(C)[-1] \xrightarrow{\beta} h_{D, \rom{c}}(D^\circ) \to h_D(D)
.
\end{equation}
\end{prop}	
	
\begin{proof}
To get the first sequence,
apply the localization sequence
\[
i_!i^! \to \id_D \to j_*j^*
\]
to $\bm1_{D}$
(\cite[Lemme~1.4.6]{ayoub_thesis_1}),
and shift it.
Then use purity in codimension $1$,
i.e., $i^! \simeq i^*\bm 1(-1)[-2]$
(\cite[Prop.~5.7]{hoyois_six} or \cite[\S~1.6.3]{ayoub_thesis_1}).
The second sequence is obtained
similarly,
using the dual localization sequence
$j^!j_! \to \id_D \to i_*i^*$.
\end{proof}

\begin{prop}\label{prop:triangle}
Let
$g \colon X \rightarrow \Aa^1$
be a quasi-projective, semistable morphism.
Recall Notation~\ref{notation:ss}.
Fix one of the branches $D := D_i$,
(and write $D_i^\circ =: D^\circ$, $v_i =: v$, etc.).
We have fibre sequences
\begin{equation}\label{eq:ss_triangle_a}
\bm1_C \to c^*\Psi_g \bm 1_{X_\eta} \xrightarrow{\rom{(a)}} \bm1_C(-1)[-1]
,
\end{equation}
and
\begin{equation}\label{eq:ss_triangle_b}
\bm1_C[-1] \xrightarrow{\rom{(b)}} c^!\Psi_g \one_{X_\eta} \to \bm1_C(-1)[-2]
,
\end{equation}
in $\DA(C)$.
In fact,
$\rom{(a)}= c_0^* \alpha \circ c_0^* (\Ay^*)$,
and $\rom{(b)}= c_0^*(\Ay^!) \circ c_0^* \beta$,
with the notation
of~\eqref{eq:computenearbycycles1}--\eqref{eq:computenearbycycles2}
and~\eqref{eq:motivetrianglea}--\eqref{eq:motivetriangleb}.
\end{prop}

\begin{proof}
For the first sequence,
use the proposition above for
$ C \xhookrightarrow{c_0} D \xhookleftarrow{v} D^\circ $
and apply $c_0^*$,
to get the sequence 
\[
\bm1_C
\to c_0^*v_* \bm1_{D^\circ}
\xrightarrow{\alpha} c_0^* c_{0*} \bm{1}_C(-1)[-1]
.
\]
Then use the fact that,
since $c_0$ is an immersion,
the counit
$ c_0^*	c_{0*} \to \id$
is invertible
(\cite[Def.~1.4.1]{ayoub_thesis_1}).
In addition,
by Cor.~\ref{cor:computability},
$u^*\Psi_g \bm1_{X_\eta} \simeq v_{*} \bm 1_{D^\circ}$,
so the middle term of the sequence gives nearby cycles at $C$:
$
c^*\Psi_g \bm1_{X_\eta}
\simeq c_0^* u^* \Psi_g \bm1_{X_\eta}
\simeq  c_0^* v_{*} \bm 1_{D^\circ}
$.
We get the first sequence,
\[
\bm1_{C}
\to c^*\Psi_g \bm1_{X_\eta}
\to \bm1_{C}(-1)[-1]
.
\] 
The second sequence is obtained similarly,
using Cor.~\ref{cor:computability} again,
so
\[
c^!\Psi_g \bm1_{X_\eta}
\simeq c_0^! u^! \Psi_g \one_{X_\eta}
\simeq c_0^! v_! \one_{D^\circ}
.
\]
\end{proof}
	
	We can now deduce that up to a constant, the variation map can be expressed in terms of the maps $\rom{(a)}$ and $\rom{(b)}$ above.

\begin{prop}\label{prop:sss_monodromy}
	Let $g$ be a special semistable morphism.
	Recall Prop.~\ref{prop:varmap}.
	The following diagram commutes:
	\begin{equation*}\begin{tikzcd}[row sep=huge]
			c^*\Psi_g\one \ar[r, "\varmap"] \ar[d, "\rom{(a)}" ']
			&[2em]  c^!\Psi_g\one(-1)  \\
			\one(-1)[-1] \ar[r, "\lambda"] 
			& \ar[u, "\rom{(b)}(-1)" '] \one(-1)[-1].
	\end{tikzcd}\end{equation*}
\end{prop}
	
\begin{proof}
	Recall the cofibre sequences~\eqref{eq:ss_triangle_a}
	and~\eqref{eq:ss_triangle_b}.
	The lemma then follows from
	the general vanishing result, Prop.~\ref{prop:neg_twist_vanishing}.
\end{proof}
	
\subsection{Recovering the monodromy of an isolated singularity}
	
\begin{prop}\label{prop:vargoodmorphism}
Let $f\colon X \to \Aff^1$ be a morphism with a good singularity
and retain the notation of Def.~\ref{definition:goodmorphism} and Notation~\ref{notation:ss}.
Suppose that $\pi \colon Y \to X$ is a proper morphism such that
$g = \pi \circ f_r$ is a special semistable morphism.
Let $\lambda$ be as in Prop.~\ref{prop:sss_monodromy}.
Let $\pi_\sigma \colon Y_\sigma \to X_\sigma$ and let $\pi_C \colon C \to p$ be the corresponding restrictions of $\pi$.
$\var_f$ is given then by
\[\begin{tikzcd}
&o^*\Psi_f\one  \arrow[d, "\simeq"]  \arrow[rr, "r/\lambda \cdot \var_f"]& &  o^!\Psi_f\one(-1) \\
& h(D^\circ)\arrow[r, "\alpha"]
&[2em] h(C)(-1)[-1] \arrow[r, "\beta(-1)"]
&[3em] h_{\rom{c}}(D^\circ)(-1) \arrow[u, "\simeq"] &
\end{tikzcd}\]
where the vertical rows are given by~\eqref{eq:computenearbycycles1} and~\eqref{eq:computenearbycycles2},
and the bottom horizontal are the natural maps $\alpha$, $\beta$ of~\eqref{eq:motivetrianglea} and~\eqref{eq:motivetriangleb},
with the closed-open embedding
$ C \hookrightarrow D \hookleftarrow D^\circ $.
\end{prop}

\begin{proof}
Substituting the result of
Prop.~\ref{prop:sss_monodromy}
into that of Prop.~\ref{prop:varcomparison},
we get a decomposition of
$r/\lambda \cdot \var_f$
as
\[
o^*\Psi_f\one	\to 
\pi_{C *} c^* \Psi_g \one \to
\pi_{C *} \one_C(-1)[-1] \to
\pi_{C *} c^! \Psi_g \one (-1) \to
o^!\Psi_f\one(-1)
\]
We now consider the first two arrows,
together with~\eqref{eq:computenearbycycles1}:
\[\begin{tikzcd}
o^*\Psi_f\one \ar[r, "\simeq"]
& o^*\pi_{\sigma *} \Psi_g \one
	\ar[d, "\simeq", " Ex_D^*"']
	\ar[r, "Ex_C^*", ]
& \pi_{C *} c^* \Psi_g \one
	\ar[dd, "\pi_{C *} \rom{(a)}"]
	\ar[dl, "\eta_{c_0}"]
\\
& \pi'_* u^* \Psi_g \one
	\ar[d, "\simeq", "\pi'_* \Ay^*"']
\\
&	\pi'_* v_* \one_{D^\circ}
	\ar[r, "\pi`_* \alpha"]
& \pi_{C *} \one_C(-1)[-1]
\end{tikzcd}\]
This diagram commutes,
as the upper triangle commutes by definition,
and the lower trapezoid commutes
by the definition of the map $\rom{(a)}$,
see Prop.~\ref{prop:triangle}. 
We can deal similarly with
the last two arrows
and then get as a result
the diagram in the proposition.
\end{proof}

\begin{rmk}
Thus, after the identifications
of~\eqref{eq:computenearbycycles1}--\eqref{eq:computenearbycycles2},
the variation of $f$ at $o$
can be described in terms of $r$, $\lambda$, and the immersions
$C \hookrightarrow D \hookleftarrow D^\circ$.
In the case of a homogeneous singularity
defined by a polynomial $F$ of degree $r$
(Def.~\ref{def:homogeneoussingularity}),
this is the immersions of projective hypersurfaces
\[
\{ T_{n+1} = 0 \}
\hookrightarrow V(F-T_{n+1}^r)
\hookleftarrow \{ T_{n+1} \neq 0 \}
\] 
in $\P^{n+1}$
(Prop.~\ref{proposition:homogeneousssred}).
Note that when $F$ is of degree $2$ this is an embedding of quadrics.

In the quasi-homogeneous case,
we have the analogous embeddings of weighted projective hypersurfaces
in $\P(\underline{a},1)$ of Prop.~\ref{proposition:qhomogeneousssred}.
We use the same notation for both cases,
and we treat them together. 
\end{rmk}

\subsection{Reduction to a one-dimensional semistable family}

We round off this section with one more reduction.

\begin{prop}[*]\label{prop:reduction_to_1dim}
In order to compute $\lambda$ in Prop.~\ref{prop:sss_monodromy},
it is enough to compute it for the following special semistable family:
\[
g' \colon Y' := \Spec k[t,x,y]/(xy-t) \to \Spec k[t] = \Affl.
\]
In particular,
$\lambda$ is a constant that does not depend on $f$
(or on the construction of a semistable model $g$ for $f$),
justifying our notation.
\end{prop}

\begin{proof}
Let $g \colon Y \to \Affl$ be a special semistable family.
As such, there is an étale map
$U \to Y$
and a smooth map
$U \to Y'$
(see, e.g.,~\cite[Lemme~3.3.36]{ayoub_thesis_2}
in a slightly different setting).
Using Prop.~\ref{prop:varsmoothcomparison}
along both these maps,
we get that the value of $\lambda$ for $g'$
is the same as for $g$. 
More details to appear in a later version.
\end{proof}	

\begin{rmk}\label{rmk:compute_param_via_real}
Illusie computes the monodromy
of the $\ell$-adic nearby cycles complex
for a semistable family~\cite[Thm.~2.6]{illusiePL}.
We can therefore
apply $\ell$-adic realization at this point,
and via
\begin{equation*}
\Q
= \End_{\DA}(\bm1(-1)[-1])
\hookrightarrow \End_{\Q_\ell}(\Q_\ell(-1)[-1])
= \Q_\ell
\end{equation*}
read off the value of $\lambda$.
Sections~\ref{sec:reduction_to_kummer}--\ref{sec:kummer}
can then be skipped,
and Thm.~\ref{thm:general_intro}
is deduced from Prop.~\ref{prop:vargoodmorphism}.
This is however not quite satisfactory from a motivic
point of view,
as Illusie's proof is built upon
an explicit $\ell$-adic construction of the nearby cycles complex
which is not available motivically.
In the next two sections,
we therefore compute $\lambda$
in an alternative way,
so that it boils down instead to the motivic
Construction~\ref{const:unip_nearby_cycles}.
\end{rmk}


\section{%
	Second reduction:
  to the nearby Kummer motive%
}\label{sec:reduction_to_kummer}

Recall the Kummer motive
$\clg{K} \in \DA(\Gm)$
from Def.~\ref{defn:Kummer}.
We call $\Psi_\id\clg{K}$ the \emph{nearby Kummer motive}.
The goal of this section is to
reduce the computation of $\lambda$
in Prop.~\ref{prop:vargoodmorphism}
to the computation of the monodromy on $\Psi_\id\clg{K}$.

The star of this section
is the family
\begin{equation}\label{eq:1dimfam}
\Spec k[t,x,y] / (xy-t) \to \Spec k[t] =: \Affl,
\end{equation}
where $ t \mapsto t $.
It is a one-dimensional,
special semistable family.
The generic fibre is $\Gm \times \Gm$,
say with parameters $t$ and $x$.
The special fibre is given by the coordinate axes of the plane,
$\Spec k[x,y]/(xy)$.
This family has two natural sections:

\begin{rmk}\label{rmk:notation_for_1dim_fam}
The sections
\begin{equation*}
t \mapsto t,
\quad x \mapsto 1,
\quad y \mapsto t
\end{equation*}
and
\begin{equation*}
t \mapsto t,
\quad x \mapsto t,
\quad y \mapsto 1
\end{equation*}
of~\eqref{eq:1dimfam},
restrict to
\begin{equation*}
\id \times 1 \colon \Gm \to \Gm \times \Gm,
\quad \text{and} \quad
\Delta \colon \Gm \to \Gm \times \Gm,
\end{equation*}
respectively,
on the generic fibre.
On the special fibre,
they restrict to
points
\begin{align*}
x_1 \colon x &\mapsto 1
& \text{and} &&
x_2 \colon x &\mapsto 0
\\
y &\mapsto 0
&&&
y &\mapsto 1,
\end{align*}
respectively.
These points $x_1$ and $x_2$,
which are identified with $1$
on the respective axes,
play an important role
later in this section.
\end{rmk}

\begin{rmk}
As the family~\eqref{eq:1dimfam}
is special semistable,
its monodromy factors as in
Prop.~\ref{prop:sss_monodromy},
and we need to compute $\lambda$.
\end{rmk}


\subsection{Computations}

We begin by
reinterpreting the Kummer motive
in terms of the family
\begin{equation*}
f \colon X := \Spec k[t,x,y]/(xy-t) \to \Affl = \Spec k[t].
\end{equation*}
To this end,
consider the following
commutative diagram of vertical cofibre sequences
in $\DA(\Gm)$:
\begin{equation*}\begin{tikzcd}[ampersand replacement=\&]
\bm1(-1)[-1]
	\ar[d, "\begin{pmatrix} 0 \\ 1 \end{pmatrix}" ']
	\ar[r, "e_\rom{K}"]
	\&[3em] \bm1
	\ar[d, "\begin{pmatrix} 0 \\ 1 \end{pmatrix}"]
	\\[3em]
\bm1 \oplus \bm1(-1)[-1]
	\ar[d, "{\begin{pmatrix} 1 & 0 \end{pmatrix}}" ']
	\ar[r, "{\begin{pmatrix} 1 & 0 \\ 1 & e_\rom{K} \end{pmatrix}}"]
	\& \bm1 \oplus \bm1
	\ar[d, "{\begin{pmatrix} 1 & 0 \end{pmatrix}}"]
	\\[2em]
\bm1 \ar[r, "\id"]
	\& \bm1 .
\end{tikzcd}\end{equation*}
As the cofibre of the first row is $\clg{K}$
(Def.~\ref{defn:Kummer}),
and the cofibre of the last row is $0$,
we have that $\clg{K}$ sits in the following cofibre sequence:
\begin{equation}\begin{tikzcd}[ampersand replacement=\&]
\bm1 \oplus \bm1(-1)[-1]
	\ar[r, "{\begin{pmatrix} 1 & 0 \\ 1 & e_\rom{K} \end{pmatrix}}"]
	\&[3em] \bm1 \oplus \bm1
	\ar[r, "{\begin{pmatrix} -c_\rom{K} & c_\rom{K} \end{pmatrix}}"]
	\&[3em] \clg{K}.
\end{tikzcd}\end{equation}
(cf.~the proof of~\cite[Lemme~3.6.41]{ayoub_thesis_2}).
We now consider an avatar of this sequence over $X_\eta = \Gm \times \Gm$.

\begin{defn}\label{defn:J}
There are natural maps
\begin{equation*}
\bm1 \to (\id\times1)_*\bm1, \quad
\bm1 \to \Delta_*\bm1,
\end{equation*}
in $\DA(\Gm\times\Gm)$.
Define $\clg{J}$ to be the cofibre of their sum:
\begin{equation}\label{eq:defJ}
\bm1 \to (\id\times1)_*\bm1 \oplus \Delta_*\bm1 \to \clg{J}.
\end{equation}
\end{defn}

\begin{rmk}\label{rmk:K_is_pushf_J}
By the above,
and the definition of the Kummer map,
we then have that $\clg{K} = \rom{pr}_{1,*}\clg{J}$.
\end{rmk}

\begin{rmk}\label{rmk:alt_J}
Denote the open complements of $\id\times1$ and $\Delta$
by $j_\rom{h}$ and $j_\rom{d}$,
respectively.
Then,
$\clg{J} = (j_{\rom{h},!}\bm1 \oplus j_{\rom{d},!}\bm1)[1]$.
\end{rmk}

\begin{prop}\label{prop:nearby_kummer_good_iso}
The map
\begin{equation*}
\beta_f \colon \Psi_\id\clg{K} \to f_{\sigma,*}\Psi_f\clg{J}
,
\end{equation*}
(from Prop.~\ref{prop:nearby_cycles_pushf_pullb})
is an isomorphism.
\end{prop}

\begin{proof}
Note that
\begin{equation*}
X = \Bl_0\Aff^2 \setminus Z,
\end{equation*}
where $Z$ is the strict transform of $\Affl \times 0$.
This indicates we should break our proof into two steps.

\emph{Step 1.}
Consider the inclusion map
\begin{equation*}
g \colon X
= \Bl_0(\Aff^2) \setminus Z
\hookrightarrow \Bl_0(\Aff^2)
,
\end{equation*}
and let $\pi$ denote the map $\Bl_0(\Aff^2) \to \Affl$.
We get a map
\begin{equation*}
\beta_g \colon \Psi_\pi g_{\eta,*}\clg{J} \to g_{\sigma,*}\Psi_f\clg{J},
\end{equation*}
and the goal of this step is to show that it is an isomorphism.
Due to the cofibre sequence \eqref{eq:defJ},
this reduces to showing that
\begin{equation*}
\beta_g \colon \Psi_\pi g_{\eta,*}\bm1 \to g_{\sigma,*}\Psi_f\bm1,
\end{equation*}
is an isomorphism.
As $\alpha_g$
is an isomorphism,
it is enough to show that
\begin{equation*}
\alpha_g^{-1}\circ\beta_g \colon
  \Psi_\pi \bm1 g_{\eta,*}g_\eta^*\bm1
  \to g_{\sigma,*}g_\sigma^* \Psi_\pi\bm1,
\end{equation*}
is an isomorphism.
Denote the inclusion of the closed complement of $g$ by
\begin{equation*}
d \colon Z \hookrightarrow \Bl(\Aff^2).
\end{equation*}
Gluing gives us fibre sequences
\begin{equation*}
d_{\star,!}d_\star^! \to \id \to g_{\star,*}g_\star^*,
\end{equation*}
(where $\star$ can be omitted or set to $\sigma$ or $\eta$)
and the natural transformations $\alpha_g$ and $\beta_g$
in Prop.~\ref{prop:nearby_cycles_pushf_pullb},
as well as their exceptional variants $\nu_d$ and $\mu_d$,
then produce a morphism of fibre sequences
\begin{equation*}\begin{tikzcd}
\Psi_\pi d_{\eta,!}d_\eta^!\bm1 \ar[d] \ar[r]
  & d_{\sigma,!}d_\sigma^! \Psi_\pi\bm1 \ar[d] \\
\Psi_\pi \bm1 \ar[d] \ar[r]
  & \Psi_\pi\bm1 \ar[d] \\
\Psi_\pi \bm1 g_{\eta,*}g_\eta^*\bm1 \ar[r]
  & g_{\sigma,*}g_\sigma^* \Psi_\pi\bm1,
\end{tikzcd}\end{equation*}
from where we reduce to showing that the topmost horizontal arrow
($\nu_d \circ \mu_d^{-1}$ in the notation of \cite[Prop.~3.1.19]{ayoub_thesis_2})
is an isomorphism.
What is left to show is therefore that
\begin{equation*}
\nu_d \colon \Psi_{\pi_Z} d_\eta^!\bm1
  \to d_\sigma^! \Psi_\pi\bm1,
\end{equation*}
where $\pi_Z \colon Z \to \Affl$
denotes the restriction of $\pi$.
We can check this locally
around $Z$,
and in particular,
we can remove the strict transform of $0\times\Affl$.
Thus, Step~1 is reduced to the following claim:

\begin{claim}
Consider 
the trivial vector bundle $\rom{pr}_1 \colon \Aff^2 \to \Affl$
and its zero section $s \colon \Affl \to \Aff^2$.
The morphism
\begin{equation*}
\nu_s \colon \Psi_{\id}s_\eta^!\bm1
  \to s_\sigma^!\Psi_{\rom{pr}_1}\bm1,
\end{equation*}
is an isomorphism.
\end{claim}

\begin{proof}[Proof of claim.]
Note that
$\alpha_{\rom{pr}_1}$ is an isomorphism,
so that it is enough to check that the composition
\begin{equation*}
\alpha_{\rom{pr}_1}^{-1} \circ \nu_s
  \colon \Psi_{\id}s_\eta^!\bm1
  = \Psi_{\id}s_\eta^!\rom{pr}_{1,\eta}^*\bm1
  \to s_\sigma^!\rom{pr}_{1,\sigma}^*\Psi_{\id}\bm1,
\end{equation*}
is an isomorphism.
But
$s_\star^!\rom{pr}_{1,\star}^*
  \simeq (\blank)\otimes\bm1(1)[2]$
is just a twist and a shift,
which the nearby cycles functor commutes with
(see~\cite[Prop.~3.1.7]{ayoub_thesis_2}).
\end{proof}

\emph{Step 2.}
We now prove the thesis
using the previous step.
Consider the blowup morphism
\begin{equation*}
h \colon \Bl_0(\Aff^2) \to \Aff^2.
\end{equation*}
It is a map of $\Affl$-schemes,
where the structure morphism from the latter to $\Affl$ is
$\rom{pr}_1 \colon \Aff^2 = \Affl\times\Affl \to \Affl$.
We get an isomorphism
\begin{equation*}
\beta_h \colon \Psi_{\rom{pr}_1} h_{\eta,*}
	\xrightarrow{\sim} h_{\sigma,*}\Psi_{\pi},
\end{equation*}
as $h$ is proper (Prop.~\ref{prop:nearby_cycles_pushf_pullb}).
There is a commutative diagram
\begin{equation*}\begin{tikzcd}
\Psi_\id f_{\eta,*} \ar[d, equals] \ar[r, "\beta_f"]
  & f_{\sigma,*}\Psi_{f} \\
\Psi_\id\rom{pr}_{1,\eta,*}h_{\eta,*}g_{\eta,*}
	\ar[d, "\beta_{\rom{pr}_1}"', "\sim" {rotate=90, anchor=north}]
  & \rom{pr}_{1,\sigma,*}h_{\sigma,*}g_{\sigma,*}\Psi_f
	\ar[u, equals] \\
\rom{pr}_{1,\sigma,*}\Psi_{\rom{pr}_1} h_{\eta,*}g_{\eta,*}
	\ar[r, "\beta_h" ', "\sim"]
  & \rom{pr}_{1,\sigma,*}h_{\sigma,*}\Psi_\pi g_{\eta,*}
	\ar[u, "\beta_g" '], \\
\end{tikzcd}\end{equation*}
and we conclude using Step 1.
\end{proof}

We need a good understanding of 
$f_{\sigma,*}\Psi_f\clg{J}$.
The rest of this subsection
is dedicated to that.

\begin{lem}
We have a fibre sequence
\begin{equation}\label{eq:pushdown_nearby_as_fibre}
f_{\sigma,*} \Psi_f \bm1
\to f_{\sigma,*} u_{1,*} v_{1,*} \bm1
	\oplus f_{\sigma,*} u_{2,*} v_{2,*} \bm1
\to f_{\sigma,*} c_*c^! \Psi_f \bm1[1].
\end{equation}
\end{lem}

\begin{proof}
This is just gluing
and Cor.~\ref{cor:computability}.
\end{proof}

\begin{prop}[*]\label{prop:1dim_computations}\ 
\begin{enumerate}
\item
There are natural isomorphisms
$f_{\sigma,*} u_{i,*} v_{i,*} \bm1
	\simeq \bm1 \oplus \bm1(-1)[-1]$.

\item
There is a natural isomorphism
$f_{\sigma,*} c_*c^! \Psi_f \bm1[1]
	\simeq \bm1 \oplus \bm1(-1)[-1]$.

\item
Under the above identifications,
both components of the right-hand map
in~\eqref{eq:pushdown_nearby_as_fibre}
are identity.
\end{enumerate}
In particular,
there is a natural isomorphism
$f_{\sigma,*} \Psi_f \bm1
	\simeq \bm1 \oplus \bm1(-1)[-1]$.
\end{prop}

\begin{proof}
1.\ \ 
Since $D_i^\circ \simeq \Gm$,
this is just Rmk.~\ref{rmk:motiveGm}.

2.\ \ 
The natural map
\begin{equation*}
f_{\sigma,*} u_{1,*} u_1^* \Psi_f \bm1
	\to f_{\sigma,*} c_*c^! \Psi_f \bm1[1],
\end{equation*}
is an isomorphism.
This will be explained in a later version.

3.\ \ 
By the choice of map in the previous part,
this is immediately true for the first component.
For the second component,
we need to show that the following diagram commutes:
\begin{equation*}\begin{tikzcd}
f_{\sigma, *} u_{2, *} v_{2, *} \bm1
	\ar[d, equals]
	\ar[r]
	& f_{\sigma, *} c_*c^! \Psi_f \bm1[1]
	\ar[r]
	& f_{\sigma, *} u_{1, *} v_{1, *} \bm1
	\ar[d, equals]
	\\
\bm1 \oplus \bm1(-1)[-1]
	\ar[rr, "\id" ']
  && \bm1 \oplus \bm1(-1)[-1] .
\end{tikzcd}\end{equation*}
The map is given by a matrix
\begin{equation*}
\begin{pmatrix}
1 & a \\
0 & b
\end{pmatrix}
\end{equation*}
where the $1$ comes from the map being an algebra morphism,
and the $0$ comes from Prop.~\ref{prop:neg_twist_vanishing}.
Since we can swap the $D_i$,
this matrix has to be its own inverse,
i.e.,
$-a/b = a$
and $1/b = b$.
Thus,
$b = \pm 1$
and $a$ is either $0$ (when $b=1$)
or arbitrary (when $b=-1$).
We will show that $b=1$ in a later version.
\end{proof}

\begin{cor}\label{cor:comp_pushf_nearby_J}
We have that
$f_{\sigma,*} \Psi_f \clg{J} = \bm1 \oplus \bm1(-1)$.
\end{cor}
\begin{proof}
We have a commutative diagram of cofibre sequences
\begin{equation*}\begin{tikzcd}[ampersand replacement=\&]
f_{\sigma, *} \Psi_f \bm1
	\ar[d, equals]
	\ar[r]
  \& f_{\sigma, *} \Psi_f (\id \times 0)_* \bm1
		\oplus f_{\sigma, *} \Psi_f \Delta_* \bm1
	\ar[d, equals]
	\ar[r]
	\& f_{\sigma, *} \Psi_f \clg{J}
	\ar[d, dashed]
	\\[1em]
\bm1 \oplus \bm1(-1)[-1]
	\ar[r, "{\begin{pmatrix} 1 & 0 \\ 1 & 0 \end{pmatrix}}" ]
  \& \bm1 \oplus \bm1
	\ar[r]
	\& \bm1 \oplus \bm1(-1) .
\end{tikzcd}\end{equation*}
\end{proof}

\begin{rmk}\label{rmk:first_comp_nearby_Kummer}
By combining
Prop.~\ref{prop:nearby_kummer_good_iso}
with Cor.~\ref{cor:comp_pushf_nearby_J},
we get an isomorphism
$\Psi_\id \clg{K} \simeq \bm1 \oplus \bm1(-1)$.
This is our first computation of the nearby Kummer motive,
out of two.
\end{rmk}

\subsection{Reduction to the nearby Kummer motive}

\begin{prop}[*]\label{prop:reduction_to_nearby_Kummer}
There is a commutative diagram
\begin{equation*}\begin{tikzcd}
\Psi_\id\clg{K}[-1] \ar[d, "N_\clg{K}" '] \ar[r]
& \bm1(-1)[-1] \ar[d, "\lambda"]
\\
\Psi_\id\clg{K}(-1)[-1]
& \ar[l] \bm1(-1)[-1]
\end{tikzcd}\end{equation*}
where the horizontal maps become the obvious ones
under the identification in
Rmk.~\ref{rmk:first_comp_nearby_Kummer}
and $\lambda$ is the parameter in Prop.~\ref{prop:reduction_to_1dim}.
\end{prop}

\begin{rmk}
Prop.~\ref{prop:reduction_to_nearby_Kummer}
thus reduces the computation of $\lambda$
to a computation of the monodromy $N_\clg{K}$ of the nearby Kummer motive.
\end{rmk}

\begin{proof}[Proof of Proposition~\ref{prop:reduction_to_nearby_Kummer}]
Consider the fibre sequence
\begin{equation*}
\Psi_f\bm1 \to x_{1,*}\bm1 \oplus x_{2,*}\bm1 \to \Psi_f\clg{J}
,
\end{equation*}
obtained by applying $\Psi_f$ to~\eqref{eq:defJ}.\footnote{%
	Compare also with Rmk.~\ref{rmk:alt_J}.%
}
Here, the points $x_i$ are introduced in
Rmk.~\ref{rmk:notation_for_1dim_fam}.
The monodromy operator gives a commutative diagram
\begin{equation*}\begin{tikzcd}
\Psi_f\clg{J}[-1] \ar[d, "{N_{\clg{J}}[-1]}" '] \ar[r]
	& \Psi_f\bm1 \ar[dl, dashed, "N'"] \ar[d, "N_{\bm1}"] \ar[r]
	& x_{1,*}\bm1 \oplus x_{2,*}\bm1 \ar[d, "0"]
	\\
\Psi_f\clg{J}(-1)[-1] \ar[r]
	& \Psi_f\bm1(-1) \ar[r]
	& (x_{1,*}\bm1 \oplus x_{2,*}\bm1)(-1)
\end{tikzcd}\end{equation*}
where the dashed arrow $N'$ exists as the monodromy operator vanishes at smooth points
(such as $x_1$ and $x_2$).
We can view $N'$ in a slightly different way:
we have seen that $N$ factors through
$c_*c^!\Psi_f\bm1(-1) \to \Psi_f\bm1(-1)$,
and there are no maps from $c_*(\blank)$ to $x_{i,*}(\blank)$,
regardless of what $(\blank)$ is;
denote the map
$c_*c^!\Psi_f\bm1(-1) \to \Psi_f\clg{J}(-1)[-1]$
by $\rom{(c)}$.
We consider the commutative diagram:
\begin{equation*}\begin{tikzcd}
\Psi_f\clg{J}[-1]
  \ar[ddd,
		bend right=24,
		start anchor = south west,
		end anchor = north west,
		"{N_{\clg{J}}[-1]}"'
		]
  \ar[d] \\
\Psi_f\bm1 \ar[d, "N_{\bm1}"] \ar[r]
  & c_*c^*\Psi_f\bm1 \ar[r]
  & c_*\bm1(-1)[-1] \ar[d, "\lambda"] \\
\Psi_f\bm1(-1)
  & \ar[ld, bend left=16, end anchor = east, "\rom{(c)}"]
		\ar[l]
		c_*c^!\Psi_f\bm1(-1)
  & \ar[l] c_*\bm1(-1)[-1] \\
\Psi_f\clg{J}(-1)[-1] \ar[u]
\end{tikzcd}\end{equation*}
from Prop.~\ref{prop:sss_monodromy}.
We apply $f_{\sigma,*}$ to it and use
Prop.~\ref{prop:nearby_kummer_good_iso},
Prop.~\ref{prop:1dim_computations},
and Cor.~\ref{cor:comp_pushf_nearby_J}:
\begin{equation*}\begin{tikzcd}
\bm1[-1]\oplus\bm1(-1)[-1]
  \ar[
		ddd,
		bend right,
		start anchor = south west,
		end anchor = north west,
		"{N_\clg{K}[-1]}"'
	]
  \ar[d] \\
\bm1\oplus\bm1(-1)[-1] \ar[d, "0"] \ar[r]
	& c^*\Psi_f\bm1 \ar[r, "\rom{(a)}"]
	&[3em] \bm1(-1)[-1] \ar[d, "\lambda"] \\
\bm1(-1)\oplus\bm1(-2)[-1]
	& \ar[ld, bend left=16, end anchor = east, "f_{\sigma,*}\rom{(c)}"]
		\ar[l]
		c^!\Psi_f\bm1(-1)
  & \ar[l, "\rom{(b)}(-1)" '] \bm1(-1)[-1]
	\\
\bm1(-1)[-1]\oplus\bm1(-2)[-1]
	\ar[u]
\end{tikzcd}\end{equation*}
This finishes the proof,
modulo the next lemma.
\end{proof}

\begin{lem}[*]\label{lemma:compatibility}
The following diagrams commute:
\begin{equation*}\begin{tikzcd}
\bm1\oplus\bm1(-1)[-1]
		\ar[d, "{(0,\id)}" '] \ar[r, "\sim"]
	& f_{\sigma,*}\Psi_f\bm1 \ar[r]
	& f_{\sigma,*}c_*c^*\Psi_f\bm1 \ar[d, "\sim"]
	\\
\bm1(-1)[-1]
&& c^*\Psi_f\bm1 \ar[ll, "\rom{(a)}"]
\end{tikzcd}\end{equation*}
\begin{equation*}\begin{tikzcd}
\bm1(-1)[-1]
		\ar[d, "{(\id,0)}" ']
		\ar[r, "\rom{(b)}(-1)"]
	& c^!\Psi_f\bm1(-1) \ar[r, "\sim"]
	& f_{\sigma,*}c_*c^!\Psi_f\bm1(-1)
		\ar[d, "f_{\sigma,*}\rom{(c)}"]
	\\
\bm1(-1)[-1]\oplus\bm1(-2)[-1]
	&& f_{\sigma,*}\Psi_f\clg{J}(-1)[-1] \ar[ll, "\sim" ']
\end{tikzcd}\quad\end{equation*}
\end{lem}
\begin{proof}
The first diagram can be redrawn as:
\begin{equation*}\begin{tikzcd}
f_{\sigma,*}\Psi_f\bm1 \ar[d] \ar[r]
  & c^*\Psi_f\bm1 \ar[d, "\rom{(a)}"] \\
\bm1(-1)[-1] \ar[r, equals] & \bm1(-1)[-1]
\end{tikzcd}\end{equation*}
We recall how the vertical maps are defined.
The leftmost map is defined starts out with
Prop.~\ref{prop:1dim_computations},
i.e.,
\begin{equation*}
f_{\sigma,*}\Psi_f\bm1 \xrightarrow{\sim} u_{1,*}v_{1,*}\bm1 \simeq q_*\bm1
\end{equation*}
(recall that $q$ denotes the structure morphism $\Gm \to \pt$).
Then, as in Rmk.~\ref{rmk:motiveGm},
we get the left vertical map
as the connecting morphism of the localization sequence
\begin{equation*}
0_!0^!\bm1 \to \bm1_{\Affl} \to j_*j^*\bm1
,
\end{equation*}
of $0$ and $\Gm$ in $\Affl$, after applying $p_*$.
On the other hand,
$\rom{(a)}$ is defined
(Prop.~\ref{prop:triangle})
as the connecting morphism of the localization sequence
\begin{equation*}
c_{0,!}c_0^!\bm1 \to \bm1_{D_1} \to v_*v^*\bm1
,
\end{equation*}
after applying $c_0^*$.
But this is the same localization sequence again
(recall that $D_1 \simeq \Affl$).
Applying the morphism
$p_* = f_{\sigma,*}u_{1,*} \to c_0^*$
to this connecting morphism yields the commutative diagram we are after.

The second diagram will be treated in a later version.
\end{proof}


\section{The nearby Kummer motive}
\label{sec:kummer}

In this section,
we compute the \emph{nearby Kummer motive} $\Psi_\id\clg{K}$,
as well as its monodromy.

Recall
the definition of the Kummer motive (Def.~\ref{defn:Kummer})
as an extension
\begin{equation}\label{eq:extension_Kummer}
\bm1 \xrightarrow{c_\rom{K}} \clg{K} \to \bm1(-1).
\end{equation}
Recall also Construction~\ref{const:unip_nearby_cycles}.
As we are interested in $\Psi_\id\clg{K}$,
we tensor the sequence~\eqref{eq:extension_Kummer}
with $\Log^\vee$,
which yields
\begin{equation}\begin{tikzcd}\label{eq:log_kummer_extension}
\Log^\vee \ar[r]
& \Log^\vee\otimes\,\clg{K} \ar[r]
& \Log^\vee(-1)
.
\end{tikzcd}\end{equation}

Denote the unit and multiplication maps of the algebra object $\Log^\vee$ by $u$ and $m$, respectively.
Let moreover $e \colon \clg{K} \to \Log^\vee$ be the natural map.
Note that the composition
$e \circ c_\rom{K} \colon \bm1_{\Gm} \to \clg{K} \to \Log^\vee$
is the unit map $u$.

\begin{lem}[{\cite[Lemme~3.6.16]{ayoub_thesis_2}}]
The first map in~\eqref{eq:log_kummer_extension} admits a retraction
given by $r := m \circ \id \otimes e$.
This induces a splitting
\begin{equation*}
\Log^\vee\otimes\,\clg{K} \simeq \Log^\vee \oplus \Log^\vee(-1).
\end{equation*}
\qed
\end{lem}

\begin{cor}\label{cor:nearby_kummer_split}
The nearby Kummer motive is computed as
$\Psi_{\id}\clg{K} \simeq \bm1 \oplus \bm1(-1)$.
\end{cor}

\begin{proof}
By the lemma,
\begin{equation*}
\Log^\vee \otimes \clg{K}
\simeq \Log^\vee \otimes (\bm1 \oplus \bm1(-1)).
\end{equation*}
Therefore,
$\Psi_\id\clg{K} \simeq \Psi_\id(\bm1\oplus\bm1(-1))$
and we conclude using
Prop.~\ref{prop:psione}.
\end{proof}

\begin{rmk}\label{rmk:nearby_Kummer_comps_agree}
The isomorphism in Cor.~\ref{cor:nearby_kummer_split}
is the same as the one described in Rmk.~\ref{rmk:first_comp_nearby_Kummer}.
To see this,
note that they both fit into a commutative diagram
\begin{equation*}\begin{tikzcd}
\bm1 \ar[d, "\id"] \ar[r]
  & \Psi_\id\clg{K} \ar[d, "\sim" {rotate=90, anchor=north}] \ar[r]
  & \bm1(-1) \ar[d, "\id"] \\
\bm1 \ar[r]
  & \bm1 \oplus \bm1(-1) \ar[r]
  & \bm1(-1),
\end{tikzcd}\end{equation*}
and thus differ by an automorphism
\begin{equation*}
\begin{pmatrix}
\id & a \\
0 & \id
\end{pmatrix} \colon \bm1 \oplus \bm1(-1) \to \bm1 \oplus \bm1(-1)
,
\end{equation*}
but $a = 0$ by Prop.~\ref{prop:neg_twist_vanishing},
and we are done.
\end{rmk}

Next, we compute the monodromy operator,
relative to this description of the nearby Kummer motive.
Recall (Construction~\ref{const:monodromy})
that
there is a map $N \colon \Log^\vee \to \Log^\vee(-1)$
which induces the monodromy operator
$N \colon \Psi_{\id} \to \Psi_{\id}(-1)$
via the functoriality of $\chi_{\id}$.
The map $N$ is induced by
$b \colon \clg{K} \to \bm1(-1)$,
and this has the following lemma as a consequence.
\begin{lem}[{\cite[85]{ayoub_etale}}]\label{lem:N_from_b}
We have that $N \circ e = u(-1) \circ b$.
\qed
\end{lem}
We record one more fact about $N$,
which is crucial in the sequel.
\begin{lem}[{\cite[Rmk.~3.8]{motivic_monodromy}}]\label{lem:N_derivation}
The map $N$ is a derivation of the algebra $\Log^\vee$,
in the sense that
\begin{equation*}
N \circ m = m(-1) \circ (N \otimes \id + \id \otimes N)
.
\end{equation*}
\qed
\end{lem}

\begin{prop}\label{prop:kummer_monodromy}
Under the isomorphism in Cor.~\ref{cor:nearby_kummer_split}
(and thus also the one in Rmk~\ref{rmk:first_comp_nearby_Kummer},
by Rmk.~\ref{rmk:nearby_Kummer_comps_agree}),
the monodromy operator $N \colon \Psi_{\id}\clg{K} \to \Psi_{\id}\clg{K}(-1)$
is given by
\begin{equation*}
\begin{pmatrix}
0 & -1 \\
0 & 0
\end{pmatrix} \colon \bm1 \oplus \bm1(-1) \to \bm1(-1) \oplus \bm1(-2).
\end{equation*}
\end{prop}

\begin{proof}
Consider the diagram:
\begin{equation*}\begin{tikzcd}
\bm1 \otimes \bm1 \ar[d, "u \otimes \id"] \ar[r, "\id \otimes c_\rom{K}"]
	& \bm1 \otimes \clg{K} \ar[d, "u \otimes \id"] \ar[r, "\id \otimes b"]
	& \bm1 \otimes \bm1(-1) \ar[d, "u \otimes \id"]
	\\
\Log^\vee \otimes \bm1 \ar[d, "N \otimes \id"] \ar[r, "\id \otimes c_\rom{K}"]
	& \Log^\vee \otimes \clg{K} \ar[l, bend left, start anchor = south west, end anchor = south east, "r"]
		\ar[d, "N \otimes \id"] \ar[r, "\id \otimes b"]
	& \Log^\vee \otimes \bm1(-1) \ar[l, bend right, start anchor = north west, end anchor =  north east, "s" ']
		\ar[d, "N \otimes \id"]
	\\
\Log^\vee(-1) \otimes \bm1 \ar[r, "\id \otimes c_\rom{K}"]
  & \Log^\vee(-1) \otimes \clg{K} \ar[l, bend left, start anchor = south west, end anchor = south east, "r(-1)"]
		\ar[r, "\id \otimes b"]
	& \Log^\vee(-1) \otimes \bm1(-1)
\end{tikzcd}\end{equation*}
Here, $s$ is the section induced by the retraction $r$ described above.
It suffices to show that the composition $r(-1) \circ N \otimes \id  \circ s$
equals minus the identity morphism,
$-\id\colon \Log^\vee(-1) \to \Log^\vee(-1)$.
We first get rid of the inexplicit section $s$.
The goal is equivalent to showing
\begin{equation*}
\varphi
:= r(-1) \circ N \otimes \id_{\clg{K}} \circ s \circ \id_{\Log^\vee} \otimes b
= -\id_{\Log^\vee} \otimes b,
\end{equation*}
as we can cancel $\id \otimes b$ on the right with $s$.
But, by definition,
$s \circ \id \otimes b = \id - \id \otimes c_\rom{K} \circ r$,
so we get that
\begin{align*}
\varphi
&= r(-1) \circ N \otimes \id \circ (
	\id
	- \id \otimes c_\rom{K} \circ r
)
\\
&= r(-1) \circ N \otimes \id
	- r(-1) \circ N \otimes \id \circ \id \otimes c_\rom{K} \circ r
.
\end{align*}
The second term can be rewritten as
\begin{equation*}
r(-1) \circ N \otimes \id \circ \id \otimes c_\rom{K} \circ r
	= N \otimes \id \circ r,
\end{equation*}
by plugging in the definition of $r(-1)$,
and using that $e \circ c_\rom{K} = u$,
and $m \circ \id \otimes u = \id$.
We deduce that
\begin{align*}
\varphi
&= r(-1) \circ N \otimes \id - N \otimes \id \circ r
\\
&= m(-1) \circ \id \otimes e \circ N \otimes \id
	- N \circ m \circ \id \otimes e
\\
&= [m(-1) \circ N \otimes \id - N \circ m] \circ \id \otimes e
.
\end{align*}
To conclude,
we invoke Lemma~\ref{lem:N_derivation},
to rewrite
\begin{align*}
m(-1) \circ N \otimes \id &- N \circ m
\\
= {}
& m(-1) \circ N \otimes \id - (m(-1) \circ N \otimes \id + m(-1) \circ \id \otimes N)
\\
= {}
& - m(-1) \circ \id \otimes N
,
\end{align*}
and from there obtain
\begin{equation*}
\varphi = - m(-1) \circ \id \otimes N \circ \id \otimes e
	= - m(-1) \circ \id \otimes u(-1) \circ \id \otimes b
	= - \id \otimes b,
\end{equation*}
by Lemma~\ref{lem:N_from_b}.
\end{proof}

\subsection{Proof of the main theorem}

We are now ready to deduce our main theorem.

\begin{thm}[The general Picard--Lefschetz formula, *]
\label{thm:abstractPL}					
Let $X$ be regular and let
$f \colon X \to \Affl$
be a flat, quasi-projective morphism.
Assume $f$ is smooth except for
an isolated (quasi-)homogeneous singularity
at $o \in X_\sigma$,
defined by
a (quasi-) homogeneous polynomial $F$
of ($\underline{a}$-weighted) degree $r$.
Let
\[
C := V(F(T_0, \dotsc, T_n)) \hookrightarrow D := V(F(T_0, \dotsc, T_n)-T_{n+1}^r),
\]
and
\[
D^\circ := D \setminus C
.
\]
That is, $C$ and $D$ are hypersurfaces in $\P^n$ and $\P^{n+1}$
(in $\P(\underline{a})$ and $\P(\underline{a}, 1)$),
respectively.
\begin{enumerate}
\item We have natural isomorphisms 
\[
o^* \Psi_f \one \simeq h(D^\circ),
\quad
o^! \Psi_f \one \simeq h_{\rom{c}}(D^\circ)
.
\]

\item
The variation $\var_f$ features in the commutative diagram 
\[\begin{tikzcd}
&o^*\Psi_f\one  \arrow[d, "\simeq"]  \arrow[rr, "-r \cdot \var_f"]& &  o^!\Psi_f\one(-1) \\
& h(D^\circ)\arrow[r, "\alpha"]
&[2em] h(C)(-1)[-1] \arrow[r, "\beta(-1)"]
&[3em] h_{\rom{c}}(D^\circ)(-1) \arrow[u, "\simeq"]
\end{tikzcd}\]
where the maps in the bottom row
are the
natural maps coming from
the localization sequences
\eqref{eq:motivetrianglea}, \eqref{eq:motivetriangleb}. 
\end{enumerate}
\end{thm}

\begin{proof}
Prop.~\ref{prop:vargoodmorphism}
gives the diagram up to a rational constant $\lambda$.
By Propositions~\ref{prop:reduction_to_1dim}
and~\ref{prop:reduction_to_nearby_Kummer},
$\lambda = -1$
is computed in
Prop.~\ref{prop:kummer_monodromy},
finishing the proof.
\end{proof}	

\section{The formula for a quadratic singularity}
\label{sec:quadratic}

In this section we assume that our base field $k$ is algebraically closed (and of characteristic different from $2$), so all the quadrics are split. We assume that $o$ in Thm.~\ref{thm:abstractPL} is a homogeneous quadratic singularity (Def.~\ref{def:homogeneoussingularity} with $r=2$), and deduce a motivic version for the classic Picard--Lefschetz formula of \cite[Exp.~XV]{SGA7} and \cite[\S~3]{illusiePL}, depending on the parity of the dimension. The objects and the maps in Thm.~\ref{thm:abstractPL} can be all computed explicitly in terms of Tate motives and maps between them, leading to the motivic formula.

\begin{notation}
	We write $\one \{-n\}$ for the motive $\one(-n)[-2n]$.
\end{notation}

\begin{prop}[Rost]\label{prop:motiveprojquadric}
	Let $Q \subset \P^{n+1}$ be a smooth quadric of dimension n. Then
	\begin{equation*}
		 h(Q) =
		 \begin{cases}
		 \bigoplus_{i=0}^n \one \{-i\} &		\text{n odd} \\
		 \bigoplus_{i=0}^n \one \{-i\}  \oplus \one \{-n/2\} & \text{n even}
		 \end{cases} .
	\end{equation*} 
\end{prop}

\begin{proof}
	These computations have been done by Rost in \cite{Rost} for Chow motives. Since effective Chow motives embed fully faithfully in Voevodsky motives (\cite[Prop.~20.1]{mvw}), this is also true in $\DA(\blank)$.
\end{proof}

\begin{defn}
	We say that $A$ is a smooth affine quadric of dimension $n$ if it is isomorphic to $Q \cap H \subset \P^{n+1} $ where $Q$ is a smooth projective quadric of dimension $n$, $H$ is a projective hyperplane, and $ C:= Q \cap H$ is smooth. Note that $C$ is then a smooth projective quadric of dimension $n-1$, embedded in $H \simeq \P^n$.
	\end{defn}

\begin{prop}[*]\label{prop:motiveaffinequadric}
	Let $A$ be a smooth affine quadric of dimension $n$, then we have the following Verdier dual fibre sequences in $\DA(k)$
		\[
	 \one  \to	 o^* \Psi_f \one \xrightarrow{m_1} \one(-\lceil n/2 \rceil )[-n] , 
	 \]
	 \[
	\bm1(-\lfloor n/2 \rfloor)[-n] \xrightarrow{m_2}	 o^!\Psi_f \one \to	 \bm1(-n)[-2n]. 
	\]
\end{prop}

\begin{proof}
	First, changing coordinates and completing to a square, we may assume that $C$ is given by a quadratic form $\phi(X_0, \dotsc, X_n)$, and  $Q$ is given by $\phi_a = \phi + aX_{n+1}^2$.
	
	We now recall Rost's proof for the motive of a quadric, which goes by induction as follows (\cite[Prop>~2]{Rost}): first we can change coordinates if necessary to write $\phi= X_0X_1 + \psi(X_2, \dotsc, X_n)$. Then consider $C_0:= \{X_0 =0\}$ inside $C$. The open complement $C \setminus C_0$ is isomorphic to an affine space and therefore its motive is trivial.
	Note that $C_0$ has the point
	$k \colon p=(0:1:\underline{0}) \hookrightarrow C_0$,
	so $k^! \one_C =  1\{-(n-2)\}$,
	and the complement $C_0 \setminus p$ is $\Aff^1$-equivalent to the quadric defined by $\psi$, which allows proceeding by induction. The same can be done with $Q$, and we get that
	\[ i^* \colon h(Q)= \one \oplus h(Q_0 \setminus p)\{-1\} \oplus \one\{-n\}  \to  h(C)= \one \oplus h(C_0 \setminus p)\{-1\} \oplus \one\{-(n-1)\} , \]
	and $i^*$ is described by the following matrix
	(see~\cite[Lemma~29]{Bachmann}):
	\[
	\begin{pmatrix}
		\id & 0 & 0 \\
		0 & i'^*\{-1\} & 0 \\
		0 & s\{-1\} & 0 
	\end{pmatrix} .
	\]
	Here $i'$ is the inclusion $C_0 \setminus p \hookrightarrow Q_0 \setminus p$, and
	$s \colon h(Q_0 \setminus p)\simeq h(\psi_a) \to \one\{-(n-2)\}$
	is the fundamental class.
	Since $i'$ is equivalent to an immersion of quadrics, induced by quadrics of dimensions lower by $2$, we may proceed by induction and get an explicit description for $i^*$.

	If $n=2m+1$ is odd,
		\[
	i^* \colon h(Q) = \bigoplus_{i=0, i \neq m}^n \one  \{-i\} \oplus \one\{-m\} \to   h(C) =\bigoplus_{i=0, i\neq m}^{n-1} \one  \{-i\}  \oplus \one \{-m\} ^{\oplus 2}
	\]
	is given by
		\[
	\begin{pmatrix}
		\id & 0 & 0 \\
		0 & 0 & 1 \\
		0 & 0 & 1 
	\end{pmatrix}
	.
	\]
	If $n=2m$ is even, then 
	\[ i^* \colon h(Q) =  \bigoplus_{i=0, i\neq m}^n \one  \{-i\} \oplus \one \{-m\} ^{\oplus 2} \to h(C) = \bigoplus_{i=0, i \neq m}^{n-1} \one  \{-i\} \oplus \one \{-m\} ,\]
		is given by
	\[
	\begin{pmatrix}
		\id & 0 & 0 & 0 \\
		0 & 0 & 1 & -1
		\end{pmatrix}
		.
		\]
	This argument will appear in more detail in a later version of this paper.		Now, using the localization sequence
	\[ h_{\rom{c}}(A) \to h(Q) \to h(C) , \] 
	we are able to compute the motive of $A$. To simplify we can cancel out similar terms according to the diagrams below. We still separate between the even and odd cases,
	here all columns rows are fibre sequences.
	When $n=2m+1$ is odd we have,
	\[\begin{tikzcd}[ampersand replacement=\&, row sep = large]
		0 \ar[r] \ar[d]	\& \bigoplus_{i=0, i\neq m}^{n-1} \one\{-i\} \ar[r, "\id"] \ar[d] \& \bigoplus_{i=0, i\neq m}^{n-1} \one\{-i\} \ar[d] \\
		h_{\rom{c}}(A) \ar[r] \ar[d] \& \bigoplus_{i=0}^n \one\{-i\} \ar[d] \ar[r]
		\& \bigoplus_{i=0}^{n-1} \one \{-i\}  \oplus \one\{-m\} \ar[d, ""] \\
		h_{\rom{c}}(A) \ar[r] \& \one\{-n\} \oplus \one\{-m\} \ar[r, "{\begin{pmatrix} 0 & 1 \\ 0 & 1 \end{pmatrix}}"] \&   \one\{-m\}\oplus  \one\{-m\}.
	\end{tikzcd}\]
		The last horizontal arrow above fits in the following commutative diagram of vertical cofibre sequences
	\begin{equation*}\begin{tikzcd}[ampersand replacement=\&]
			\bm1\{-n\}
			\ar[d, "\begin{pmatrix} 0 \\ 1 \end{pmatrix}" ']
			\ar[r, "0"]
			\&[3em] \bm1\{-m\} 
			\ar[d, "\begin{pmatrix} 0 \\ 1 \end{pmatrix}"]
			\\[3em]
			 \bm1\{-m\} \oplus \one\{-n\}
			\ar[d, "{\begin{pmatrix} 1 & 0 \end{pmatrix}}" ']
			\ar[r, "{\begin{pmatrix} 1 & 0 \\ 1 & 0 \end{pmatrix}}"]
			\&  \bm1\{-m\} \oplus \one\{-m\}
			\ar[d, "{\begin{pmatrix} 1 & -1 \end{pmatrix}}"]
			\\[2em]
			 \bm1\{-m\} \ar[r, "\id"]
			\&  \bm1\{-m\} .
	\end{tikzcd}\end{equation*}
	
	Taking the fibre of the first row and shifting, gives the fibre sequence 
	\[
	\one(-m)[-2m-1] \to	h_{\rom{c}}(A) \to  \one(-n)[-2n],
	\]
	and by Verdier duality, as $h(A) = \mathbb{D}(h_{\rom{c}}(A)) (n)[2n]$, we get
	\[
\one \to	h(A) \to  \one(-m)[-n].
\]	
	When $n=2m$ is even, we have
	\[\begin{tikzcd}
		0 \ar[r] \ar[d]	&\bigoplus_{i=0, i\neq m}^{n-1} \one \{-i\} \ar[r, "\id"] \ar[d] & \bigoplus_{i=0, i\neq m}^{n-1} \one\{-i\} \ar[d] \\
		h_{\rom{c}}(A) \ar[r] \ar[d] &  \bigoplus_{i=0}^n \one \{-i\} \oplus \one\{-m\} \ar[d] \ar[r]
		& \bigoplus_{i=0}^{n-1} \one \{-i\}  \ar[d, ""] \\
		h_{\rom{c}}(A) \ar[r] & \one\{-n\} \oplus \one\{-m\} \oplus \one\{-m\} \ar[r] &   \one\{-m\} .
	\end{tikzcd}\]
	Similarly to the odd case, we get
	\[
	h_{\rom{c}}(A) \simeq 	\one(-m)[-2m] \oplus  \one(-n)[-2n] ,
	\]
	and can deduce $h(A)$, so the statement of the proposition holds.
\end{proof}

\begin{thm}[%
	The quadratic Picard--Lefschetz formula, *%
]\label{thm:quadratic_PL}
Assume $k$ is algebraically closed of characteristic different from $2$.
Let $X$ be regular and let
$f \colon X \to \Aa^1_k$
be a flat, quasi-projective morphism of relative dimension $n$,
smooth except for
an isolated quadratic singularity $o \in X_\sigma$.	
If $n$ is even, then the map $\var$ is the zero map.
If $n=2m+1$ is odd,
$-\var$ factors as
\begin{equation*}
o^*\Psi_f \one
\xrightarrow{m_1} \bm1(-m-1)[-n]
\xrightarrow{m_2(-1)} o^!\Psi_f \one(-1), 
\end{equation*}
with notation as in Prop.~\ref{prop:motiveaffinequadric}.
\end{thm}

\begin{proof}
	When $n$ is even, the result follows immediately from Prop.~\ref{prop:motiveaffinequadric}, applying Prop.~\ref{prop:neg_twist_vanishing}.
	
	For the odd case consider the following diagram:
	\[
	\begin{tikzcd}[row sep=large]
		\one \arrow[d] & & \one(-n)[-2n] \\
		h(A) \arrow[r,"\alpha"] \arrow[d,"m_1"'] 
		& h(C)(-1)[-1]  \arrow[dr, dashed] \arrow[r,"\beta(-1)"] 
		& h_{\rom{c}}(A)(-1) \arrow[u]  \\
		\one(-m-1)[-n] \arrow[rr,"?"] \arrow[ur, dashed] & & \one(-m-1)[-n] \arrow[u,"m_2(-1)" '] .
	\end{tikzcd}
	\]
	The columns are the fibre sequences of Prop.~\ref{prop:motiveaffinequadric}.
	The middle row is the factorization of $-2\var_f$ according to our main result, Thm.~\ref{thm:abstractPL}.
	The lower horizontal map exists using
	Prop.~\ref{prop:neg_twist_vanishing}
	with the vertical sequences.
	Similarly, we get the existence of the diagonal arrows, as 
	\[h(C)(-1)[-1] \simeq \bigoplus_{i=0}^{n-1} \one (-i-1)[-2i-1]  \oplus \one [-m-1](-n) \]
	by Prop.~\ref{prop:motiveprojquadric}. We can compute the diagonal arrows, as the maps $\beta$ and $\alpha$ come from the sequence considered in Prop.~\ref{prop:motiveaffinequadric} and its dual, then deduce $?=2$, and so the statement of the theorem. A more detailed proof will appear in a later version of this paper.
\end{proof}	

\printbibliography

\end{document}